\documentclass[leqno,twoside,11pt]{amsart}

\usepackage{latexsym, url, color, amssymb}
\usepackage{mathrsfs}
\usepackage{cancel}
\usepackage{esint}
\def\XXint#1#2#3{{\setbox0=\hbox{$#1{#2#3}{\int}$ }
\vcenter{\hbox{$#2#3$ }}\kern-.6\wd0}}

\theoremstyle{plain}

\DeclareMathOperator*{\esssup}{ess\,sup}

\newtheorem{theorem}{Theorem}
 \newtheorem{proposition}{Proposition}[section]
\newtheorem{corollary}[proposition]{Corollary}
\newtheorem{lemma}[proposition]{Lemma}
\theoremstyle{definition}
\newtheorem{remark}[proposition]{Remark}
\newtheorem{definition}{Definition}

\def\bees{\begin{equation*}}
\def\eees{\end{equation*}}
 \def\bee{\begin{equation}}
\def\eee{\end{equation}}
 
\numberwithin{equation}{section}
\numberwithin{figure}{section}

\def\Det{{\mathcal{D}et}\,}
\def\cc{{\rm curl}^T {\rm curl}\,} 
 
\def\R{{\mathbb R}}
\def\e{\varepsilon} 
\def\g{\mathscr{G}} 
\def\n{\vec n} 
\def\E{\mathscr{E}}
\definecolor{Green}{rgb}{0, 0.65,0}
\definecolor{Random}{rgb}{0.7, 0.0,0.9}

\begin{document}

\title[The geometry of  $C^{1,\alpha}$  isometric immersions] 
{The  geometry of $C^{1,\alpha}$  flat isometric immersions}
\author{Camillo De Lellis and Mohammad Reza Pakzad}
\address{ Camillo De Lellis, School of Mathematics, Institute for Advanced Study, 1 Einstein Dr., Princeton NJ 05840, USA} 
  \email{camillo.delellis@math.ias.edu}

\address{Mohammad Reza Pakzad, Université de Toulon, CS 60584, 83041 TOULON CEDEX 9, France}
 \email{pakzad@univ-tln.fr}
 
\date{}
 
\begin{abstract}
We show that any isometric immersion of a flat plane domain into $\R^3$ is developable provided it   enjoys the little H\"older regularity $c^{1,2/3}$. In particular, isometric immersions of  local $C^{1,\alpha}$  regularity with $\alpha>2/3$ belong to this class. The proof is based on the existence of a weak notion of second fundamental form for such immersions, the analysis of the Gauss-Codazzi-Mainardi equations in this weak setting, and a parallel result on the very weak solutions to the degenerate Monge-Amp\`ere equation analyzed in \cite{lepamonge}.
\end{abstract}
\maketitle
 
\section{Introduction}
  
 \subsection{Background}
 
 It often happens, in differential geometry, as in other contexts where nonlinear PDEs appear, that the solutions to a given nonlinear system of equations satisfy also systems of higher order equations carrying significant geometric or analytic information.  A very basic example is the Gauss-Codazzi-Mainardi equations satisfied by any smooth isometric immersion of a 2 dimensional Riemannian manifold into $\R^3$.  Note that, in their turn,   these equations can be used to prove certain rigidity statements for the immersion, e.g.\@ to show that any isometric immersion of a flat domain is developable, i.e.\@ roughly speaking, at any given point $p$ its image contains a segment of $\R^3$ passing through $p$.
 
 The subtlety which is often glossed over in the context of classical differential geometry, where all the mappings are assumed 
 smooth, or at least twice or thrice continuously differentiable, is that,  due to nonlinearity,  the passage from say a given set of first order equations, e.g.\@  the system of isometric immersion equations in our example, to the higher order equations requires a minimum of regularity. There is no guarantee that in the absence of this regularity the geometric information hidden in the higher order equations would be accessible.  The above mentioned subtlety could be a source of confusion.  Coming back to our example, it is often assumed that all isometric images of flat domains must be ruled, which is  untrue. Already, at the dawn of modern differential geometry, no less figures than Lebesgue and Picard intuited the existence of a surface which could be  flattened onto the plane with no distortion of relative distances, but nowhere contained a straight  segment. To quote Picard\footnote{\cite[Page 555]{picard},  translated and quoted in \cite{cajori}.}: 
 
 \begin{quote}
 According to general practice, we suppose in the preceding analysis, as in all infinitesimal geometry of curves and surfaces, the existence of derivatives which we need in the calculus. It may seem premature to entertain a theory of surfaces in which one does not make such hypotheses. However, a curious result has been pointed out by Mr Lebesgue  (Comptes Rendus, 1899 and thesis)\footnote{\cite{leb}. See also \cite[Pages 164-165]{friedman} and the discussion therein of Lebesgue's heuristic construction. }; according to which one may, by the aid of continuous functions, obtain surfaces corresponding to a plane, of such sort that every rectifiable line of the plane has a corresponding rectifiable line of the same length of the surface, nevertheless the surfaces obtained are no longer ruled. If one takes a sheet of paper, and crumples it by hand, one obtains a surface applicable to the plane and made up of a finite number of pieces of developable surfaces, joined two and two by lines, along which they form a certain angle. If one imagines that the pieces become infinitely small, the crumpling being pushed everywhere to the limit, one may arrive at the conception of surfaces applicable to the plane and yet not developable [in the sense that there is no envelope of a family of planes of one parameter and not ruled].
 \end{quote}  

 
The  suggested constructions by Lebesgue and then Picard  fall short of being totally satisfactory as they point to continuous but yet not differentiable mappings. However, they were perhaps the heralds of a celebrated result due to Kuiper \cite{kuiper}, who based on the work by John Nash \cite{Nash1, Nash2} proved that any Riemannian manifold can be $C^1$ isometrically embedded into a Riemannian manifold of one dimension higher. (Actually, by Nash and Kuiper's construction,  $C^1$ isometric embeddings can approximate any short immersion). As a consequence, the flat disk can be $C^1$-isometrically embedded into any  given small ball of  $\R^3$ and hence the image cannot be ruled.   To make the matters worse, these embeddings can be made through recurrent self-similar corrugations \cite{borrelli1, borrelli2}, whose images seem, without a proof at hand, not to include any straight segment of $\R^3$ in any scale.  In contrast, and returning to our original developability statement in a higher regularity setting where the Gauss-Codazzi equations hold true,  Hartman and Nirenberg \cite{HaNi} proved, incidentally almost at the same time  Nash and Kuiper showcased their results,  that a $C^2$-regular isometric immersion of a flat domain must have a locally ruled surface as its image. We have come full circle, and this  discussed dichotomy leaves us with a fundamental question any analyst would like to ask: {\bf At what regularity threshold or thresholds between $C^1$ and $C^2$ the unruled isometries transition into the ruled  regime?} 

\subsection{Recent developments}

To put the problem described above in a broader context, we note that the above dichotomy -known in the recent literature as {\it flexibility} vs.\@ {\it rigidity}-  is not specific to the case of  isometric immersions. The Nash and Kuiper scheme for creating highly oscillatory anomalous solutions of typically lower  regularity  through iterations could be studied under the broader topic of convex integration \cite{gromov, spring} for the differential inclusions (or PDEs re-cast in this framework), and the involved  density or flexibility related results  are usually referred to as h-principle. The existence of such  h-principle is usually accompanied by a parallel rigidity property which indicates that the construction cannot be carried out in high regularity scales.  A notable example is the recent discovery by the first author and  Sz\'ekelyhidi \cite{DS}: They showed that the system of Euler equations in fluid dynamics is given to the convex integration method   and that Nash and Kuiper's iterations can be adapted in this case in order to create non-conservative compactly supported continuous flows approximating an open set of subsolutions (see \cite{Sz, DS-Bul} for a thorough discussion of the connection between the two problems of turbulence and isometric immersions).  Their results, and the subsequent work in improving the regularity of  the anomalous solutions  stood in contrast with the rigidity result reflected in the conservation of energy for solutions passing a certain regularity threshold ($\alpha>1/3$ for $C^{1,\alpha}$-H\"older continuous solutions) in \cite{CoTi, eyink}, as conjectured by Onsager \cite{onsager}. Their approach finally lead to the complete resolution of this conjecture by Isett in \cite{isett}; see  \cite{DS-survey} for a review of the history of the problem and the intermediate results. 
 
Coming back to the question of isometric immersions,  a parallel question is whether the images of isometric immersions into $\R^3$ of 
closed convex surfaces (i.e.\@ $2$-manifolds with no boundary and positive Gaussian curvature)  are rigid motions. This result follows from the convexity of the image of such surfaces. This convexity fails for $C^{1,\alpha}$ isometric immersions of convex surfaces if $\alpha<1/5$,  in which regime the above mentioned h-principle for isometric immersions holds true, as shown in \cite{1/5, SzCa2}, improving on the results by \cite{bori2, CDS}. On the other hand, Conti, De Lellis and Sz\'ekelyhidi proved in \cite{CDS, CDS-errata} that when $\alpha>2/3$, the isometric image of a closed convex surface  is convex, (from which it follows that the immersion is a  rigid motion), and that more generally the h-principle cannot hold true for isometric immersions of any elliptic $2$-manifold with or without boundary. Mischa Gromov conjectures in  \cite[Section 3.5.5.C, Open Problem 34-36]{gromov2} that  the  transition threshold is $\alpha=1/2$ in this case.   The best evidence so far for this conejcture is provided in   \cite{DeIn}, where the authors show that $C^{1,1/2}$ regularity could be the borderline regularity for the identity of intrinsic covariant derivative of the $2$-manifold and its geometric interpretation. 

This article concerns the fundamental question of the sub-$C^2$ regularity scales for which, given an isometric immersion of a flat domain, one is able to make sense of fundamentally geometric objects such as the second fundamental form and the Gauss-Codazzi equations. The final goal is to show that  the developability property of such isometric immersions  (as defined below in Definition \ref{defdev} and proven for $C^2$ isometries in  \cite{HaNi})  survives for $C^{1, \alpha}$   isometric immersions  if $\alpha>2/3$. As we shall see, parts of the project, e.g.\@ the definition of the second fundamental form, can be carried over for $\alpha>1/2$, but our analysis falls short of proving a weak variant of the Codazzi-Mainardi equations   for  $1/2<\alpha< 2/3$. In this manner, the question of the optimal threshold remains open in the flat case as it is for the elliptic case. 

Before proceeding, it should be mentioned that the developability of flat isometric immersions have been discussed in the literature in other regularity regimes.  Pogorelev \cite[Theorem  1, p. 603]{Po 73} proves this developability under the rather weak assumptions of  $C^1$ regularity and the existence and vanishing  of the total extrinsic curvature (defined as a measure). This result lies at the background of our conclusion through the analysis made in \cite{lepamonge} and \cite{Kor, Kor2}. On the other hand parallel results have been shown by the second author for $W^{2,2}$ isometric immersions \cite{Pak}, where a slight $C^1$ regularity gain in this case is also proven, and by Jerrard in \cite{J2010} for the class of Monge-Amp\`ere functions \cite{fu1}. The former result was based on observations made by Kirchheim on solutions to the degenarate Monge-Amp\`ere equation \cite[Chapter 2]{Ki-hab}, and lead to a subsequent statement regarding the smooth density of isometric immersions. These statements were generalized  to higher dimensional domains for the co-dimension 1 case in \cite{LP}. The latter approach by Jerrard \cite{Je1, J2010}, anchored in geometric measure theory, opened the path for proving $C^1$ regularity and a full range of developability results for isometric immersions of any dimension and co-dimension of sub-critical Sobolev regularity \cite{JP}.

\subsection{The main result}
 
 We first begin by a definition to clarify the notion of developabilty as we understand it. 
 
  \begin{definition}\label{defdev}
  Let $\Omega\subset \R^2$ be a domain and  $u: \Omega \to \R^k$ be $C^1$-smooth. We say $u$ is developable  if for all $x\in \Omega$, either $u$ is affine in a neighborhood of $x$, or there exists   a line  $L_x$ containing $x$ such that  the Jacobian derivative $\nabla u$ is constant on $l_x$, the connected component of $x$ in $L_x \cap \Omega$. 
 \end{definition}
  
  \begin{remark}
For bounded domains, $l_x$ is always a segment whose both ends lie on $\partial \Omega$.  We allow for unbounded domains, and hence $l_x$ could be a complete line, a ray, or an open segment. With some abuse of notation, we will refer to all of them  as \lq\lq segments" throughout the paper. 
     \end{remark}
    
There are other equivalent formulations for the developability whose equivalence we will show further on in Propostion \ref{loc=glo}.  In particular, local and global developability, as laid out in Corollary \ref{loc=glo-really}, are equivalent. Also, from Corollary \ref{dirunq2} and Lemma \ref{dirunq}   it follows that when $u$ is developable, the segments $l_x$ are uniquely determined outside the maximal affine region; they do neither intersect each other nor pass through the local constancy regions of the Jacobian derivative.

 In this article  we prefer to work with the little H\"older spaces. For  our choice of notations and precise definitions of these spaces  see Section \ref{toolbox}. Working with the little H\"older scale pinpoints  the  needed control -for the effectiveness of our arguments- on the oscillations of the mapping gradient.  The corresponding results in the classical H\"older regime is then obtained as a corollary.  
 
The main result of this article is the following theorem:

   \begin{theorem}\label{iso23}
 Let $\Omega\subset \R^2$ be a domain and $2/3 \le \alpha<1$. If $u \in c^ {1,\alpha}(\Omega, \R^3)$ is an isometric immersion, then $u$ is developable.  In particular,  if $\alpha>2/3$,  we have $C^{1,\alpha}(\Omega) \subset c^{1,2/3} (\Omega)$ and therefore all isometric immersions $u : \Omega \to \R^3$ of  local $C^ {1,\alpha}$ regularitDy are developable. 
 \end{theorem}     
     
    \begin{remark}
    Based on the observations made in this article, the proofs of \cite{CDS, CDS-errata} for the rigidity statements are also valid under little H\"older $c^{1,2/3}$ regularity assumption. A more difficult task would be to  consider other types of regularity regimes, e.g.\@ the fractional Sobolev one. For this direction, See \cite{LPS21} and the  discussions therein.
 \end{remark}
     
Per \cite[Theorem  1, p. 603]{Po 73}, what we need  to show is that the image of the Gauss map $\vec n$ is of measure zero.  Following \cite[Corollary 5]{CDS} it is possible to prove that the  Brouwer degree ${\rm deg} (y, V  ,\vec n)$ of   $\vec n$ of the immersion vanishes for all open sets $V\subset \Omega$ and $y\in \R^2$, when defined. However this is not sufficient  to conclude with the needed statement $|\vec n(\Omega)|=0$ for developability:  Through a similar approach as in \cite[Section 5]{MalyMartio}, for each $\alpha\in (0,1)$,  one can construct  a map   in $C^{0, \alpha}(\Omega, \R^2)$ whose local degree vanishes everywhere, but whose image is onto the unit square. Hence the zero degree Gauss map could still have an image of  positive measure and this obstacle necessitates another strategy.   Indeed, our argument  uses a slight improvement on the parallel result proved for the degenerate {\it very weak Monge-Amp\`ere equation} in \cite{lepamonge}, which  uses the degree formula for \underline{both} the gradient of  the solution, and  its  affine perturbations of positive degree, as follows:

The very weak Hessian determinant \cite{Iwa, FM}  of a given
function $v\in W^{1,2}_{loc}(\Omega)$ is defined to be 
$$ \Det D^2v :=- \frac12 \Big (\partial_{11} (\partial_2 v)^2+ \partial_{22} (\partial_1 v)^2 -2 \partial_{12} (\partial_1 v \partial_2 v)\Big ) = -\frac 12 \cc  (\nabla v \otimes \nabla v) $$ 
in the sense of distributions. The operator $\Det D^2$ coinsides with the usual Monge-Amp\`ere operator $\det \nabla^2$ 
when $v\in C^2$ and $\nabla^2 v$ stands for the Hessian matrix field of $v$. (A $C^3$ regularity is needed for a straightforward calculation). The following statement regarding the degenerate Monge-Amp\`ere equation is a crucial ingredient of our analysis (Compare with  \cite[Theorem 1.3]{lepamonge}).

\begin{theorem}\label{mp23dvp}
Let $\Omega \subset \R^2$ be an open domain and let
$v\in c^{1, 2/3} (\Omega)$ be such that   
$$ \Det D^2 v =0, $$ 
in $\Omega$. Then $v$ is developable.
\end{theorem} 
    
  \begin{remark}
Note that per Corollary \ref{loc=glo-really} the assumption $v\in C^{1,\alpha} (\overline \Omega)$ in \cite[Theorem 1.3]{lepamonge}
can too be easily relaxed to the local  regularity $v\in C^{1,\alpha}(\Omega)$  in any bounded or unbounded domain.   
\end{remark} Theorem \ref{mp23dvp} can be proved following the exact footsteps of Lewicka and the second author in  proving \cite[Theorem 1.3]{lepamonge}, taking into account Lemma \ref{estimcvls} and Proposition \ref{ltl0} below in showing the following degree formulas for any open set $U\subset \Omega$:
\begin{equation}\label{degree}
\begin{array}{ll}
\forall y\in \R^2 \setminus \nabla v(\partial U)  & {\rm deg} (\nabla v, U, y) =0,  \quad \\ 
\forall \, 0<\delta \ll 1 \quad \forall y\in F_\delta (U) \setminus F_\delta (\partial U) &  {\rm deg} (F_\delta, U, y) \ge 1, 
\end{array}
\end{equation} where $F_\delta(x_1, x_2):= \nabla v(x_1, x_2 ) + \delta (-x_2, x_1)$. See \cite[Equations (7.6) and (7.9)]{lepamonge}. 
The rest of the proof remains unchanged. We will leave the verification of details to the reader.

 The main strategy of our proof of Theorem \ref{iso23} is hence  to relate the given isometric immersion $u\in c^{1,2/3}$ to a scalar function $v$ which satisfies the assumptions of Theorem \ref{mp23dvp}: In a first step, we show that a notion of the second fundamental form $A$ in the weak sense exists for immersions of   $C^{1,\alpha}$-regularity for $\alpha>1/2$.  Next, the Codazzi-Mainardi equations is used to 
 prove that if $u$ is isometric, $A$ is  curl free, implying that it must be the Hessian matrix of a scalar function, namely  the sought after function $v$.
We then  show the required regularity for $v$, and proceed to  prove using the Gauss equation  that $v$ satisfies the very weak degenerate Monge-Amp\`ere equation as required by the assumptions of  Theorem \ref{mp23dvp}. Finally, we need to  prove that developability of $v$, as derived from Theorem \ref{mp23dvp}, implies the developability of the isometric immersion $u$. Apart from Theorem \ref{mp23dvp}, the $2/3$-H\"older exponent regularity is only required for proving  that a weak version of Codazzi-Mainardi equations holds for isometric $u \in c^{1,2/3}$, i.e.\@ when we need to show that $A$ is curl free.

The article is organized as follows: In Section \ref{developability}, we will  present and prove a few statements regarding the developability properties of $C^1$ mappings. In Section \ref{toolbox}, we will gather a few analytical tools which deal with properties of H\"older continuous functions and with quadratic differential expressions in terms of functions of low  regularity. In particular, our basic proposition \ref{lem1} will allow us to define a second fundamental form for the immersions of H\"older type  regularity. The subsequent section is dedicated to the definition and properties  of this second fundamental form. In Section \ref{thmproof}, building on the previous sections, we   complete the proof of Theorem  \ref{iso23}.   Appendix \ref{component} is dedicated to a side result (Proposition \ref{compdev}) on the developability of each component of the immersion, which can be  shown independently with a shorter proof.  Finally, in Appendix \ref{litmod}  the proof of some standard facts regarding little H\"older spaces are presented.  

\subsection {Acknowledgments.}
This project was based upon work supported by the National Science Foundation and was partially funded by the Deutsche Forschungsgemeinschaft (DFG, German Research Foundation).  M.R.P. was supported by the NSF award DMS-1813738 and by the DFG  via SFB 1060 - Project Id  211504053.   The second author would like to personally thank  Stefan M\"uller  for the opportunity  to work at IAM, University of Bonn, Germany, where part of this work was completed, and for his kind support.

  \section{Preliminaries: Developable mappings}\label{developability}

We have gathered in this section a few statements we will need, and their proofs, regarding the developable mappings in two dimensions.   Most of the material in this section are well known and can be found in one form or another in the   literature on the topic \cite{HaNi, Po 73, ChL, Ki-hab,  VWKG, Pak, MuPa, Kor, Kor2, Ho1, Ho2, J2010, LP, JP, lepamonge}, but the authors  do not know of any instance where the following statements are explicitly formulated in concert.  As already observed in \cite{Ho1}, much of the developability properties of a mapping concern the level sets of  its Jacobian derivative, so we formulate our statements having the mappings  $f= \nabla u$ in mind, where $u$   is a developable mapping. 

\subsection{Global vs.\@ local developability} Developability  can be defined by local or global formulations,  which turn out to be equivalent.  There is a risk of confusion, which must be avoided, due to slight differences between the possible formulations (see e.g.\@ the footnote on p. 875 in  \cite{Kor2}).  Also, it would be useful to have a set of equivalent statements at hand to streamline the arguments. The following proposition states three equivalent conditions which the  Jacobian derivative of a  $C^1$ mapping can satisfy to be developable:

\begin{proposition}\label{loc=glo} Let $\Omega\subset \R^2$ be any domain and let $f:
\Omega \to \R^m$. If $f$ is continuous, then  the following three conditions are equivalent. 

\begin{itemize}
\item[(a)] For all $x\in \Omega$, either $f$ is constant in a neighborhood of $x$, or there exists a line  $L$ containing $x$ such that $f$ is constant on the connected component of $x$ in $L\cap \Omega$.

\medskip 
\item[(b)] For all $x\in \Omega$, there exists an open disk $B_x \subset \Omega$ centered at $x$ with the following property:  
To all  $p\in B_x$, we can associate a line $L_p$ containing $p$ such that  $f$ is constant on the segment $\tilde l_p =L_p\cap B$.  Moreover, for any $y,z\in B_x$, $\tilde l_y \cap \tilde l_z = \emptyset$, or $\tilde l_y= \tilde l_z$. 
\medskip
\item[(c)] For all $x\in \Omega$, there exists an open disk $B_x$ centered at $x$ in $\Omega$ such that 
for all  $p\in B_x$, there exists a line $L \subset \R^2$ containing $p$ so that $f$ is constant on the segment $L\cap B_x$. 
\end{itemize}
\end{proposition}

\begin{corollary}\label{loc=glo-really}
$u\in C^1(\Omega, \R^k)$  is developable if and only if  for any point in $x\in \Omega$, there exists a neighborhood $V$ of $x$ on which the restriction $u|_{V}$ is developable.   
\end{corollary}

\begin{remark}
Local regularity assumptions are sufficient in the assumptions of Theorems \ref{iso23} and \ref{mp23dvp} because of 
the equivalence of global and local developability as formulated above.   
\end{remark} 

\begin{remark}\label{exm}
The continuity of $f$ is necessary for the implications ${\rm (a)} \Rightarrow {\rm (b), (c)}$ in Proposition \ref{loc=glo} as the following example demonstrates. The function 
$$
f:\R^2 \to \R, \quad f(s,t):= \begin{cases} 0  & st \neq 0 \\ 1 & st=0, \end{cases} 
$$ satisfies condition (a) but not (b) or (c).  Also, the geometric configuration of the level sets of $f$ in \cite[Figure 1]{Pak} shows the necessity of the distinction between the points in the local constancy regions from other points in condition (a). In other words, a global formulation of condition (b) is not a statement equivalent to (b): A continuous function  on a convex domain $\Omega$ can satisfy conditions (a) and (b) while there are points $x\in\Omega$ such that  for no line $L$ containing $x$ the function is constant on the intersection $L\cap \Omega$.
 \end{remark}
 
 \begin{remark}
In the  recent literature on similar problems, global (and hence inequivalent) versions of condition (b) are formulated as \lq \lq developability of $f$''  in \cite[Definition 2.29]{Ki-hab} and \lq \lq $\Omega$-developability of $f$ on $\Omega$" in \cite[Section 3.1]{Ho1}. On the other hand condition (a) appears in the statement of \cite[Theorem II]{Pak}, and then in \cite{MuPa}  as the \lq \lq  condition (L)". It  is also equivalent to $f$ being \lq \lq countably developable"  according to  \cite[Definition 1]{Ho1} and \cite{Ho2}. (Note that \cite[Assumption (24)]{Ho1} is redundant for continuous $f$; see Lemma \ref{dirunq}).   Finally,  it is condition (c) that is stated as the property of the gradients (with empty-interior images) of  
$C^1$ mappings in the main result of \cite{Kor, Kor2}. (This local result was an ingredient of the proof of Theorem \ref{mp23dvp} in \cite{lepamonge}).  Generalizations of (a)  to weaker regularity or higher dimensional settings can be found in \cite{VWKG, J2010, JP}.
\end{remark}

\begin{remark}
The developability of a $C^1$ mapping  on $\Omega$, as formulated in Definition \ref{defdev}, translates therefore to  its Jacobian derivative  satisfying condition (a). To clarify some discrepancy in the literature, we emphasize that we prefer to reserve the term  {\it developable} for those mappings whose  Jacobian derivatives
are  constant along lower dimensional flat strata in one form or another \cite{LP, JP}, rather than directly using it for their 
derivative mappings as done in \cite[Definition 2.29]{Ki-hab} or \cite{Ho1, Ho2}. This is for historical reasons,  since this term  is conventionally  used to refer  to (ruled)  smooth surfaces of zero Gaussian  curvature \cite{doCar}. 
\end{remark}
\begin{proof} $ $

\medskip 

{(a)} $\Rightarrow$  {(b)}: 

This is implicitly proved  in \cite[Lemma 3.7]{LP}. We pursue 
a slightly different strategy.   For any $f: \Omega \to \R$, let  $C_{f}$  be the set on which $f$ is locally constant: 
 $$
 C_{f} := \{ x\in \Omega;\,\,  f \,\, \mbox{is  constant in a neighborhood of}\,\, x\}.
 $$   Note that $f$ is constant on the connected components of the open set $C_{f}$.  
 If $f$ satisfies the condition (a) of the proposition, for each $x\in \Omega \setminus C_f$ we choose 
 a line $L_x$ so that  $f$ is constant on the connected component $x$ in $L_x\cap \Omega$, denoted by $l_x$. 
 
We begin by the  following simple but useful lemma.  
 
 \begin{lemma}\label{segs}
 Let $f: \Omega \to  \R^m $ be continuous  satisfy condition {\rm (a)}  in Proposition \ref{loc=glo}.  Let $\Delta \subset \Omega $ be a  closed triangular  domain. 
 If $f$ is constant on  two edges of $\Delta$, then $f$ is constant on $\Delta$. 
 \end{lemma}
 
 \begin{proof}
Let  $x,y,z$ be the vertices of $\Delta$ and assume  that $f$ is constant on $[x,y]$ and $[x,z]$.   For any point $w\in  \Delta \setminus C_f$, $l_w$ must cross one of the two segments $[x,y]$ or $[x,z]$. This implies $f(w) =f(x)$.  
  If $w\in \Delta\cap C_f$,  and if the segment $[w,x] \subset C_f$, then once again $f(w)=f(x)$. Finally, if  $[w,x] \not\subset C_f$,  let  $\tilde w$ be the  closest element to $w$ on  the segment $[w,x]$  not in $C_{f}$.  We already proved that  $f(\tilde w) = f(x)$. But since $\tilde w$ is in the closure of the connected component of $w$ in $C_{f}$, and $f$ is continuous, we have also $f(w) = f(\tilde w) = f(x)$. We conclude that $f|_\Delta  \equiv f(x)$.
\end{proof}
 
 We now make a crucial observation on the constancy directions $L_x$. 

 \begin{corollary}\label{dirunq2}
 Let $f$ be continuous and satisfy condition {\rm (a)}  in Proposition \ref{loc=glo}.  If $x\in \Omega\setminus C_f$,  there exists only one line $L$ for which  $f$ is constant on the connected component $l$ of $x$ in $L\cap \Omega$. 
 \end{corollary}  
 \begin{proof}
 
 Let $x\in\Omega \setminus C_f$  and assume  that for two lines $L \neq \Lambda$, $f$ is constant on the connected components $l$ and $\lambda$ of $x$ in respectively $L\cap \Omega$ and $\Lambda \cap \Omega$. Choose a disk $B(x,\rho) \subset B(x, 2\rho) \subset \Omega$ and note that $\partial B(x,\rho)$ intersects $l$ and $\lambda$ on 4 points $x_1, x_2 \in l$ and $\chi_1, \chi_2 \in \lambda$. Now  the assumptions of Lemma \ref{segs} is satisfied for 
the four closed triangular domains with vertices $x, x_i, \chi_j$, $i,j\in \{1,2\}$. This implies that $x\in C_f$, which contradicts the first assumption on $x$.  
 \end{proof}
 
As a consequence, when $f$ is continuous, for each $x\in \Omega \setminus C_f$, there is only one choice for $L_x$ and there is no ambiguity in the notation. 
We further observe:

 \begin{lemma}\label{dirunq}
  Let the continuous mapping $f:\Omega \to \R^m$  satisfy condition {\rm (a)}  in Proposition \ref{loc=glo}.  Then for all $x\in \Omega\setminus C_{f}$, $l_x \subset   \Omega\setminus C_{f}$. Moreover, for all $y, z\in \Omega \setminus C_{f}$, $l_y \cap l_z = \emptyset$, or $l_y = l_z$.
 \end{lemma}

\begin{proof}
For showing the first conclusion,  we argue by contradiction: Let $y \in C_{f} \cap l_x$ 
 and let $V$ be a an open neighborhood of $y$ in $\Omega$ on which $f$ is constant.    
 We observe that for $\delta >0$ small enough  $B(x,\delta) \setminus l_x$ cannot entirely lie in $C_{f}$, since otherwise, 
$B(x,\delta) \setminus l_x$  having only two connected components on both sides of $l_x$, and $f$ being continuous, the value of $f$ would be constant on  $B(x,\delta)$, contradicting  $ x \notin C_{f}$.  Therefore, there is a sequence $x_k \in \Omega \setminus C_{f}$ converging to $x$ such that  $f(x_k) \neq f(x)$.  Since the value of $f$ on $x_k$ and $x$ differ, $l_{x_k} \cap l_x =\emptyset$ for all $k$. As $x_k\to x$, we deduce that  $l_{x_k}$ must locally 
uniformly converge to $l_x$. This implies that $l_{x_k} \cap V$ is non-empty for $k$ large enough and hence $f (x_k) = f(y) = f(x)$, which is a contradiction.  
 
To prove the second statement, also assume by contradiction that $l_y \neq l_z$ and that $x\in l_y \cap l_z  \neq \emptyset$. Choose a disk $B(x,\rho) \subset B(x, 2\rho) \subset \Omega$ and note that $\partial B(x,\rho)$ intersects $l_y$ and $l_z$ on 4 points $y_1, y_2 \in l_y$ and $z_1, z_2 \in l_z$. Now  the assumptions of Lemma \ref{segs} is satisfied for 
the four closed triangular domains with vertices $x, y_i, z_j$, $i,j\in \{1,2\}$. This implies that $x\in C_f$, which contradicts the first statement, as $y\notin C_f$ but $x\in l_y \cap C_f$.  
 \end{proof}

We are now ready to prove (b) assuming (a). We first note that the condition (b)  is obvious if $x\in C_f$. If $x\notin C_f$, without loss of generality,  and through rotations, dilations and translations, we can assume that $x= (0,0)$ and that $l_x$ is parallel to  the horizontal axis.  We show that we can find $\delta^-, \delta^+> 0$ such that the condition (b) holds true  
for $B^+(x,\delta^+)$ and $B^-(x, \delta^-)$ in the upper and lower planes, where the open half disks are defined by 
$$
B^\pm (x,\delta):= B(x, \delta) \cap \{ (s,t ) \in \R^2; \,\, \pm t  >0\}. 
$$ Then we can choose $B= B(x, \min\{\delta^+, \delta^-\})$ and the proof is complete.  

Without loss of generality we concentrate on the upper half plane. 
If there exists $\delta^+ >0$ such that $B(x,\delta^+) \subset C_f$,  the conclusion is obvious: We choose $L_y$ always parallel to the horizontal axis.  Otherwise there exists a sequence $x_k \notin C_f$ in the upper open half plane converging to $x$.  Note that by Lemma  \ref{dirunq},  any two (possibly unbounded) segments $l_y, l_z$ do not intersect within $\Omega$ unless they are the same. This implies that  $l_{x_k}\cap l_x = \emptyset$ and that the $l_{x_k}$ must converge locally uniformly to $l_x$.  We choose  $\rho>0$ such that the closed rectangular 
box 
$$
S_\rho^+ := \{ (s,t)\in \R^2; \,\, |s| \le \rho,  0\le t \le \rho\}
$$
is a subset of  $\Omega$. We let $x^{\pm}_0:= (\pm\rho, 0)$, $x^\pm = (\pm \rho, \rho)$ and choose $k$ large enough such that for $y=x_k$,  $l_y= l_{x_k}$ intersects the two  segments $[x_0^\pm, x^\pm]$ in their respective interior points $y^-$ and $y^+$. The closure of the convex open quadrilateral $P$ created by the four vertices $x^-_0, x^+_0, y^+, y^-$ lies in $S^+_\rho$ 
and we have that   $[y^-, y^+] \subset l_y$ and $[x^-_0, x^+_0] \subset l_{x}$.

We claim that the condition (b) is valid for $P$ (standing for $B_x$), which completes the proof since there exists some $\delta^+>0$, such 
that $B^+(x, \delta^+)\subset P$. For any $z\in P$, if $z\not \in C_f$, then we can choose $\tilde l_z = l_z \cap P$, which lies between $\tilde l_{x_0}$ and $\tilde l_y$. If, on the other hand, $z\in C_f$,  let the points $z_1 \in  [y^-, y^+]$ and $z_0 \in [x^-_0, x^+_0]$ be those points on $\partial P$ 
which lie vertically above and below $z$. Since $z_0, z_1 \notin C_f$, we can choose the two closest elements of $P\setminus C_f$ to $z$ on the segment $[z_0, z_1]$, which we name respectively by $\tilde z_0, \tilde z_1$. We observe that  $l_{\tilde z_0}$ and $l_{\tilde z_1}$ cannot intersect $l_x$ and $l_y$, which contain the upper and lower boundaries of $P$, and hence will intersect  the two vertical boundaries of $P$ between  $x^\pm_0$ and $y^\pm$ in $\tilde z^\pm_i$, $i=0,1$.  
We claim that $f$ is constant  in the region $X$ enclosed between $l_{\tilde z_0}$ (which is  above $\tilde l_x$ and below $l_{\tilde z_1}$) and $l_{\tilde z_1}$ (which is above $\tilde l_{\tilde z_0}$ and below $\tilde l_y$) in $P$. Note that $X$ is the convex quadrilateral created by  $\tilde z^-_0, \tilde z^+_0, \tilde z^+_1, \tilde z^-_1$ and contains $z$ in its interior. 

The latter claim about the constancy of $f$ on $X$, if proven, completes the proof of the former claim regarding $P$, since we are free to choose our non-intersecting constancy segments in this region $X$ (in particular for $z$) in a manner that no such two segments intersect within $P$: If the directions  $l_{\tilde z_0}$ and $l_{\tilde z_1}$ are parallel, choose the line $L_p$ for all points  $p\in X$ to be the line  parallel to them passing from $p$. If, on the other hand, the extensions of $l_{\tilde z_0}$ and $l_{\tilde z_1}$ meet at a point $q$ (outside of $\Omega$),  for all $p\in X$ we choose the line $L_p$ to be the one passing through $p$ and $q$.   If $\tilde l_p$, for any $p\in X$, were to intersect any other  $\tilde l_{\tilde p}$ for $\tilde p\in P \setminus X$, it would have to first intersect one of the  two $\tilde l_{\tilde z_i}$, $i=0,1$, which does not happen by construction. Note that $X$  is  the connected component of $C_f\cap P$ containing $z$. We can therefore  foliate $P$ in constancy segments by exhausting  all the connected components of $C_f \cap P$. 

To prove our final claim, that is the constancy of $f$ on the region $X$, we note that $f$ is constant on the segment  
$\tilde l_{\tilde z_1} =[\tilde z^-_1, \tilde z^+_1]$,  and on the segment $[\tilde z_0, \tilde z_1]$, whose interior lies in $C_f$. Since $\tilde z_1 \in [\tilde z_1^-, \tilde z_1^+]$, by Lemma \ref{segs}, $f$ must be constant in the triangular domain created by  by $\tilde z_0, \tilde z^-_1, \tilde z_1^+$. But then this last assertion implies that $f$ is constant on the segments
$[\tilde z_0,  \tilde z^-_1]$ and $[\tilde z_0, \tilde z^+_1]$, which combined with Lemma \ref{segs} again, using the constancy 
of  $f$ on segments $[\tilde z_0, \tilde z^-_0]$ and $[\tilde z_0, \tilde z^+_0] \subset l_{\tilde z_0}$, proves the claim that $f$ is constant on $X$.
The proof of (a) $\Rightarrow$ (b) is complete.

  \medskip

{(b)} $\Rightarrow$ {(c)}:

 It is logically obvious.

    \medskip
    
{(c)} $\Rightarrow$  {(a)}: 

The proof can be found in the last page of  \cite{lepamonge}. 
We reproduce it here for completeness. Let $x\in \Omega$ and consider a Ball $B_x$ over which the condition (c) is satisfied. In particular, $f$ is constant on $L_x\cap B_x$ for some line $L_x$ passing through $x$.  Now, consider the (possibly infinite) segment
$$
c_x:= \,\, \mbox{The connected component of} \,\, x \,\, \mbox{in} \,\,  \{ y\in L_x \cap \Omega; f(y) = f(x)\}. 
$$ We  claim that $c_x$, which contains the segment $L_x\cap B_x$, is either equal to 
$l_x$, i.e.\@  the connected component of $x$ in  $L_x\cap \Omega$, or $x\in C_f$. This proves  that $f$ is constant either on $l_x$, or in a neighborhood of $x$, and so (a) holds true. 

We assume that  $c_x\neq l_x$, and prove $x\in C_f$. In this case $c_x$ must admit at least one endpoint within $\Omega$. Let $y \in \Omega$ be that endpoint. Consider the open disk $B_y$ centered at $y$ in which all points admit constancy segments within  $B_y$. Let the line $L_y$ passing through $y$ be such that $f$ is constant on the segment $L_y \cap B_y$ and  let $z,w$ be the endpoints of this segment. $L_y$ and $L_x$  cannot be parallel, since then $c_x$ can be extended 
along $L_x = L_y$ to include either $z$ or $w$, which contradicts the fact that it must be a maximal connected 
constancy subset of $l_x$. Therefore, we can choose an element $\tilde y \in B_y \cap c_x$ to form an open  triangle $\Delta$ with vertices $\tilde y, z, w$. $f$ is constant on $[y,z]$ and on $[\tilde y, y]$, where $y\in (z,w)$. 
Since no segment $L_p \cap B_y$ departing from a point $p\in \Delta$  can reach $\partial B_y$ on its both ends without crossing $[w,z]$ or $[\tilde y, y]$, we deduce that $f |_\Delta \equiv f(y) = f(x)$. 

We observe that  $C_f$  contains $\Delta$,  and hence also  the open  segment $(\tilde y, y)$. We have thus found a non-empty interval in $c_x$ from which we can propagate our local constancy property and so reach to the desired  conclusion of $x\in C_f$. We argue by contradiction: If $x\notin C_f$, let $\tilde x$ be the closest point to $y$  on $[x, y]$ which is not in $C_f$. Certainly $C_f \supset (\tilde x , y) \supset (\tilde y ,y)$ and $(x,y) \supset (\tilde x, y)$. Consider once more the open disk $B_{\tilde x}$ according to the condition (c).  
Since $C_f$ is open, for some $\hat y\in (\tilde x, y)$ close enough to $\tilde x$,  $f$ is constant on a segment  $[\hat z, \hat w]\subset B_{\tilde x} \cap C_f$ orthogonal to $c_x$ at $\hat y$, with 
$\hat y \in  (\hat z , \hat w)$.  Also note that for $\hat x \in [x, \hat y)$,  $f$ is constant  on the portion  $[\hat x, \hat y] = [x, \hat y] \cap  B_{\tilde x} \supset [\tilde x, \hat y]$ of the segment $[x, \hat y] \subset c_x$.  Since $f$  takes the value $f(x)$ on both of  $[\hat x, \hat y]$ and $[\hat z, \hat w]$, both within the set $B_{\tilde x}$, a similar argument as above shows that $f$ is constant on 
the open triangle $\Hat \Delta$ with vertices $\hat x, \hat z, \hat w$, and so $\hat \Delta\subset C_f$. But  then $\tilde x \in (\hat x, \hat y) \subset \hat \Delta \subset C_f$, which contradicts its choice. 
  \end{proof}

In view of Proposition \ref{loc=glo},  and as already observed in \cite[Proposition 2.30]{Ki-hab}, we have: 

\begin{corollary}\label{lipvec}   Let $\Omega \subset \R^2$ be a domain and let the continuous mapping $f: \Omega \to  \R^m$  
satisfy any of the equivalent conditions in Proposition \ref{loc=glo}. Then for all $x\in \Omega$, there exists $L>0$,  an open disk $B$   centered at  $x$, and a  unit Lipschitz  vector field $\vec \eta :  B  \to \R^2$ with ${\rm Lip}\,\, \vec \eta \le L$,   for which
 $$
\displaystyle \forall y\in B \quad   \vec \eta (y) = \vec \eta (y+ s \vec \eta (y))  \,\,\, \mbox{{\rm and}}  \,\,\, 
    f (y) = f(y+ s \vec \eta (y)) \,\,\, \mbox{{\rm for all}} \,\, s \,\, \mbox{{\rm for which}}  \,\,\, y+ s  \vec \eta (y) \in B. 
$$   
\end{corollary}

 \begin{proof}
We choose a disk $B_x = B(x, 2\delta) \subset  \Omega$  according to the condition (b) and we let  $B=B(x, \delta)$.
For any $p\in B$, there exists a line $L_p$ such that $f$ is constant on $\tilde l_p = L_p \cap B_x$, and no two such lines 
meet within $B_x$. We define a the mapping ${\bf \Lambda }: B\to \R\mathbb P^1$ which associates to each point $p\in B$ the element of the real projective line
$\R\mathbb P^1$ determined by the direction of $L_p$, and we note that it is constant along the segments $L_p\cap B$ and continuous.   Since $B$  is simply connected, ${\bf \Lambda }$ can be lifted to a continuous mapping  
$\vec \eta : B \to \mathbb S^1$. By constancy of ${\bf \Lambda }$ along the segments $L_p\cap B$, $\vec \eta$ can only take two distinct values along  them, and so its continuity implies that $\vec \eta$ is constant along these directions, which are now determined by  $\vec \eta$ itself. Since two distinct lines $L_p$ and $L_q$, for $p,q\in B$, do not meet  except possibly at an at least a $\delta$-distance from $\partial B$, we conclude that ${\rm Lip} \,\vec \eta \le 1/\delta$.   
  \end{proof} 
 
\subsection{ Developability through test functions} 
The following lemma will allow us to translate the developability properties of $v\in C^{1}$  
into a property for the distribution  $\nabla^2 v$:

\begin{lemma}\label{weakconst}
 Let $B\subset \R^2$ be an open disk and  $f :B \to \R^m$ be continuous. Also, let    $\vec \eta: B\to \R^2$ be a unit Lipschitz vector field  such that the following property holds true:
  $$
 \displaystyle \forall x\in B \quad   \vec \eta (x) = \vec \eta (x+ s \vec \eta (x))  \,\,\, \mbox{{\rm for all}} \,\, s \,\, \mbox{{\rm for which}}  \,\,\, x+ s  \vec \eta (x) \in 
B. 
$$ Then the following two properties are equivalent: 
 
\begin{itemize} 
\item[(a)] $\displaystyle  \displaystyle \forall x\in B \quad  f (x) = f (x+ s \vec \eta (x)) \,\,\, \mbox{{\rm for all}} \,\, s \,\, \mbox{{\rm for which}} \,\,\, x+ s  \vec \eta (x) \in 
B$, 
\item[(b)]
 $ \displaystyle  \forall \psi\in C^\infty_c(B) \quad \int_B f {\rm div}  (\psi \vec \eta) =0.$ 
\end{itemize} 
\end{lemma}
 
\begin{proof}
 The   condition (b)  states that the distributional derivative of  $f$ in direction of the vector field $\vec \eta$ vanishes. The proof shows that a regular enough change of variable reduces the problem to the case when $\vec \eta$ is constant.  
 
 Let $\vec \xi = \vec \eta^\perp$  and $x \in B$.  Choose $T_0>0$ such that $\gamma: (-T_0,T_0) \to B $ is a solution to the ODE 
$ \gamma' (t) = \vec\xi (\gamma(t))$ in $B$ with initial value $\gamma(0) =x$. Let $L = {\rm Lip} \,\vec \eta$,  $k(t)= -\gamma''(t) \cdot \vec \eta(\gamma(t))$ and note that  $\|k\|_{L^\infty(-T_0, T_0)} \le L$. 
It is straightforward to see that there exists $0<t_0\le T_0/2$ such that for all $t, t'\in [-t_0,t_0]$, 
and all $s, s' \in \R$, 
$$
 \gamma(t) + s \vec \eta (\gamma (t)),  \gamma(t') + s' \vec \eta (\gamma (t'))\in B \implies   \gamma(t) + s \vec \eta (\gamma (t)) \neq  \gamma(t') + s' \vec \eta (\gamma (t')), 
$$ unless $t=t'$ and $s=s'$. Indeed, assume by contradiction that there exist sequences $t_k, t'_k\to 0$, $s_k, s'_k$, such that for all $k$, $(t_k, s_k) \neq (t'_k, s'_k)$, and 
$$
 \gamma(t_k) + s_k \vec \eta (\gamma (t_k)) = \gamma(t'_k) + s'_k \vec \eta (\gamma (t'_k)) \in B. 
$$ Note that $t_k = t'_k$ implies $s_k = s'_k$, and hence we must have $t_k \neq t'_k$ for all $k$, implying on its turn that $\gamma(t_k) \neq \gamma(t'_k)$ for all $k$.  On the other hand, by the main property of $\vec \eta$ we obtain
$$
\vec \eta(\gamma (t_k)) = \vec \eta (\gamma (t_k) + s_k \vec \eta (\gamma (t_k)) =  \vec \eta (\gamma (t'_k) + s'_k \vec \eta (\gamma (t'_k)) = 
 \vec \eta(\gamma (t'_k)). 
$$ This yields
$$
\gamma(t_k) - \gamma(t'_k) = (s'_k - s_k) \vec \eta (\gamma (t_k)),  
$$ implying that 
$$
\frac{\gamma(t_k) - \gamma(t'_k)}{{t_k - t'_k} }  \cdot \vec\xi (\gamma (t_k)) =0.  
$$ Passing to the limit we obtain $|\vec \xi(x)|^2 = \gamma'(0) \cdot \vec \xi (\gamma(0)) =0$,  which is a contradiction. 

Since $\gamma([-t_0, t_0])$  is compact in $B$, we choose $0<s_0<\frac1{2L}$ such that the image of  $[-t_0, t_0] \times [-s_0, s_0]$ under   
$$
\Phi(t,s) := \gamma(t) + s \vec \eta (\gamma (t)),  \,\, t\in [-t_0,t_0], s\in [-s_0, s_0],
$$ lies compactly in $B$. We note that $\Phi$ is one-to-one and Lipschitz on $U:= (-t_0,t_0) \times (-s_0, s_0)$ and that   
we have
$$
\det \nabla \Phi (t,s) = 1+ sk(t) \ge \frac 12  \,\,\, \mbox{a.e. in} \,\, U.
 $$  We conclude that $\Phi :U \to \Phi(U)$ is a bilipschitz change of variable.  
 
We let $V_x= \Phi(U)$, which is an open neighborhood of $x$ in $B$.  We calculate for any $\psi\in C^\infty_c(V_x)$:
$$
{\rm div} (\psi \vec \eta) = \partial_{\vec \eta} (\psi \vec \eta) \cdot \vec \eta + \partial_{\vec \xi}(\psi \vec \eta) \cdot \vec \xi
= \nabla \psi\cdot \vec \eta + \psi \partial_{\vec \xi} \vec \eta \cdot \vec \xi, 
$$ which gives for $\tilde \psi = \psi \circ \Phi$:
$$
{\rm div} (\psi \vec \eta)  (\Phi (t,s))= \partial_s \tilde \psi + \tilde \psi \frac{k(t)}{1+sk(t)}. 
$$ Therefore, letting $\tilde f (t,s) = f \circ \Phi(s,t)=  f(\gamma(t) + s \vec \eta (\gamma(t))$ we obtain
\begin{equation}\label{chv}
\begin{aligned}
\int_{V_x} f {\rm div}  (\psi \vec \eta) & = \int_U f{\rm div}(\psi \vec \eta) \circ \Phi (s,t) (1+ s k(t)) \,\,dsdt  \\ & = \int_U \tilde f (s,t) 
\Big (\partial_s \tilde \psi  + \tilde \psi \frac{k(t)}{1+sk(t)}\Big) (1+ sk(t)) \, ds dt 
\\ &= 
\int_U \tilde f(t,s) \Big ( k(t) \tilde \psi  + \partial_s \tilde \psi + s k(t) \partial_s \tilde \psi \Big ) \, dsdt 
\\ & = \int_U \tilde f(t,s) \partial_s ((1+ sk(t))\tilde \psi)\, dsdt.
\end{aligned}
\end{equation} If (a) is satisfied, then $\tilde f(t,s) = \tilde f(t)$ for all values of $s$ for which the function is defined, and hence, since $\tilde \psi$ is compactly supported in $U$, we 
obtain:
\begin{equation}\label{local}
\forall \psi \in C^\infty_c(V_x) \quad \int_{V_x} f {\rm div}  (\psi \vec \eta) =0.
\end{equation}   

Now for  $\psi \in C^\infty_c(B)$, we let $K:= {\rm supp}(\psi) \subset B$. $K$ is compact and so it admits a finite covering of open sets  of the form $V_{x_i}$, for $i=1, \cdots, n$. 
We consider a partition of  unity associated with this covering
$$
\displaystyle \theta_i \in C^\infty_c(V_{x_i}), \quad \sum_{i=1}^n \theta_i =1 \,\, \mbox{on}\,\, K,
$$ and we conclude with (b) by \eqref{local}:
$$
\int_{B} f {\rm div}  (\psi \vec \eta) = \sum_{i=1}^n \int_B f {\rm div} (\theta_i \psi \vec \eta) = \sum_{i=1}^n \int_{V_{x_i}} f {\rm div} (\theta_i \psi \vec \eta) =0.
$$

To prove the converse, assume (b) and let  
 $\varphi \in L^1 (U)$,  be such that 
\begin{equation}\label{sav0}
\int_{-s_0}^{s_0} \varphi(t,s)\,ds =0,  
 \end{equation} and define 
$$ \displaystyle \phi(t,s'):=\frac{1}{1+ s'k(t)} \int_{-s_0}^{s'} \varphi (t,s) \, ds.$$  We can construct 
  a sequence  $\psi_k \in C^{0,1}_c(V_x)$ such that 
    $$
  \psi_k \circ \Phi \to \phi \quad \mbox{in} \,\, L^1 (U) \quad \mbox {and}\quad \partial_s (\psi_k \circ \Phi) \to \partial_s \phi  \quad \mbox{in} \,\, L^1 (U). 
  $$
Therefore by (b)
$$
\int_{V_x} f {\rm div}  (\psi \vec \eta) =0 \,\, \forall \psi\in C^\infty_c(V_x) \implies 
\int_{V_x} f {\rm div} (\psi_k \vec \eta) =0,
$$ which in view of \eqref{chv}    implies
$$
\int_U \tilde f  \varphi  \, ds dt = \int_U \tilde f (t,s) \partial_s ((1+ sk(t) ) \phi)  \, ds dt = \lim_{k\to \infty} \int_U \tilde f (t,s) \partial_s \Big ((1+ sk(t))  \psi_k \circ \Phi \Big )  \, ds dt = 0.
$$ This fact being true for all $\varphi \in L^1(U)$ satisfying \eqref{sav0} implies that 
$$  \forall t\in (-t_0, t_0)\quad \forall s \in (-s_0, s_0)  \quad \tilde f(t, s) = \fint_{-s_0}^{s_0} \tilde f(t,s) \, ds. $$ This means that for $|s|<s_0$:
$$
f(x+ s\vec \eta (x)) = \tilde f(0, s) = \tilde  f(0,0) = f(x). 
$$  The global property (a) is  a direct consequence of this local property
around each $x\in B$, and the fact that $\vec \eta$ itself is constant on  $x+ s\vec \eta (x)\in B$ for all $s$.
\end{proof}

 \section{A H\"older continuity toolbox}\label{toolbox}
 
 \subsection{Notations}  
For any open set $U \subset \R^n$, sufficiently differentiable function  $f:U \to \R$, and nonnegative integer $j$ we will denote by $[\cdot]_{j;U}$
the supremum norm of its $j$th (multi-index) derivatives over $U$ and by 
$$
\|f\|_{k; U}:=  \sum_{j=0}^k [f]_{j;U}
$$  its $C^k$ norm.   For $0<\alpha \le 1$, the corresponding H\"older seminorms and norms  are identified  by the standard  conventions
 $$[f]_{k,\alpha; U}:=  \sup_{\begin{array}{c} x,y \in  U\\ x\neq y \end{array}}   \frac{|D^kf(x)-D^k f(y)|}{|x-y|^\alpha}, 
 $$
$$ \|f\|_{k,\alpha; U}:= \|f\|_{k; U} + [f]_{k,\alpha;U}.
$$   The modulus of continuity of $f: U \to \R$ is defined by
 $$
\omega_{f;U}(r):=  \sup_{\begin{array}{c} x,y \in  U \\  0< |x-y| \le r \end{array}}   |f(x)-f(y)|. 
$$ We also introduce the $\alpha$-H\"older modulus of $f$ for $0<\alpha<1$: 
$$
[f]_{0,\alpha;U|r}:=  \sup_{\begin{array}{c} x,y \in  U \\  0< |x-y| \le r \end{array}}   \frac{|f(x)-f(y)|}{|x-y|^\alpha}. 
$$
We will drop the subscript $U$ when denoting these quantities for the specific case of an open set named $U$, hence 
$\|\cdot\|_{\cdot}=  \|\cdot\|_{\cdot;U}$, etc.

For a fixed mollifying kernel $\phi \in C^\infty_c(B(0,1))$ with $\int\phi=1$, and $x\in \R^n$, we denote   
the standard convolution of a given mapping  $f$ with $\phi$ over the  length scale $\e$ by 
$$f_\e(x):= f\ast \phi_\e (x)= \int_{\R^n} f(x-y) \phi_\e(y) dy,$$
where for all $x$
$$
\phi_\e(x):= \e^{-n} \phi(\frac x\e).
$$ Naturally $f_\e(x)$ is defined for $\e$ small enough provided $f$ is integrable in some neighborhood of $x$. 

Throughout this article, the universally bounded constant $c=O(1)$ might change but is independent of all the data, unless specified by an argument.  We will also use the usual little-$o$ convention, i.e.\@ for any two functions $b, \tilde b  : [0, \e_0) \to \R^+$, we have:
$$ 
\tilde b (\e) \,\, \mbox{is}\,\,  o(b(\e)) \Longleftrightarrow    \forall \e \quad \tilde b (\e)  \le c b(\e) \quad \mbox {and} \quad \lim_{\e\to 0^+} \frac{\tilde b(\e)}{b(\e)} =0.  
$$ 

We will repeatedly refer to the following convolution estimates  in our arguments:
 
\begin{lemma}\label{estimcvls}
Let $V\subset U\subset \R^n$ be open sets such that ${\rm dist}{(\overline V, \partial U)} >0$. Assume that $f,h: U\to \R$ are locally integrable.    Then for  any $0<\alpha<1$ and  $\e < {\rm dist} (\overline V, \partial U)$:
\medskip
\begin{itemize}
\item[(i)] $\|f_\e - f\|_{0;V} \le  c[f]_{0,\alpha|\e} \e^\alpha$,
\medskip
\item[(ii)] $\|\nabla f_\e\|_{0;V} \le c[f]_{0,\alpha|\e} \e^{\alpha-1}$,
\medskip
\item[(iii)] $\|f_\e h_\e - (fh)_\e\|_{1;V} \le c[f]_{0,\alpha|\e} [h]_{0,\alpha|\e}\e^{2\alpha-1}$.
\end{itemize}
\end{lemma} \noindent  
Note the basic estimate
$$
\forall x \in V \quad \forall y\in B(0,\e) \quad  |f(x) - f(x-y)| \le \omega_{f;U}(\e) \le   [f]_{0,\alpha|\e} \e^\alpha.
$$  Based on this estimate, the proof follows the same lines as in \cite[Lemma 1]{CDS} and is left to the reader.  The commutator estimate Lemma \ref{estimcvls}-(iii)  will be  crucial for our analysis. It will be used to prove the quadratic estimate \eqref{gest} 
(\cite[Proposition 1]{CDS}) for the pull-back metric of mollified isometric immersions. It  is also a necessary ingredient of the proof of Theorem \ref{mp23dvp} through establishing the degree formulas \eqref{degree} for $v\in c^{1,2/3}$, following \cite{lepamonge}.

 \subsection{Little H\"older spaces}
 
Our goal here is to define and characterize the elements of the little H\"older spaces $c^{1,\alpha}$.   We denote by $C^k(U)$ the functions which are  continuously differentiable up to the $k$th order in $U$, and by $C^k(\overline U)$  the functions whose derivatives up to order $k$ in $U$ exist and admit a continuous extension to $\overline U$. For any open bounded weakly Lipschitz domain  $U\subset \R^n$, we recall the definition of the H\"older space    
$$
C^{k,\alpha}(\overline U):= \{f\in C^k(U);\, \|f\|_{k,\alpha;U} < \infty\}.
$$     By uniform continuity of derivatives, all the derivatives of $f$ up to the order $k$ have continuous extensions to $\overline U$ and so $ C^{k, \alpha} (\overline U) \subset C^k(\overline U)$.   For an arbitrary open set $\Omega \subset \R^n$  we say $f\in C^{k,\alpha}(\Omega)$ if and only if for any point  $x\in \Omega$, $f\in C^{k,\alpha}(\overline U)$ for some open neighborhood   $U\subset \Omega$ of $x$. 
\begin{remark}
 We have restricted our attention to weakly Lipschitz bounded domains $U$ in order to insure the inclusion $C^1(\overline U) \subset C^{0,\alpha}(\overline U)$ and to avoid unnecessary technical complications in what follows. 
 \end{remark}

 \begin{definition}\label{deflit}
 Let $0<\alpha< 1$ and $U\subset \R^n$ be an open bounded weakly Lipschitz domain. The little H\"older space $c^{0,\alpha}(\overline U)$ is the closure of $C^1(\overline U)$ with respect to  $\|\cdot\|_{0,\alpha;U}$ in $C^{0,\alpha}(\overline U)$.  For any open set $\Omega \subset \R^n$,  we say $f\in c^{0,\alpha}(\Omega)$ if and only if for any point  $x\in \Omega$, $f\in c^{0,\alpha}(\overline U)$ for an open neighborhood   $U\subset \Omega$ of $x$.  For $k\in \mathbb N$, the space $c^{k,\alpha}(\overline U)$ (resp.\@ $c^{k,\alpha} (\Omega)$)   is defined to be the set of all functions $f \in C^k(U)$ (resp.\@ $C^k(\Omega)$)  for which  the components of $D^k f$ belong to $c^{0,\alpha}(\overline U)$ (resp.\@ $c^{0,\alpha}(\Omega)$).
\end{definition}

 \begin{remark}
Under the above regularity assumptions on $U$, it can be shown that $c^{k,\alpha}(\overline U)$ is the closure of $C^\infty(\overline U)$ in $C^{k,\alpha}(\overline U)$.  We will not use this fact in our arguments.
 \end{remark}
  
  It is useful to identify the elements of the little H\"older spaces  in terms of their modulus of continuity. In this regard we state  the following  well known equivalence, whose proof is presented for the sake of  completeness (see also Appendix \ref{litmod}).
   
 \begin{proposition}\label{ltl0}  Let  $U$ be a   bounded weakly Lipschitz domain, $0<\alpha< 1$.   Then the following statements are equivalent. 
 \begin{itemize}
 \item[(i)]  $f\in c^{0, \alpha}(\overline U)$.
\item[(ii)]  $f: U\to \R$ admits an extension $\tilde f\in  c^{0,\alpha} (\R^n)$ and $\tilde f_\e|_U \to f$ in $C^{0,\alpha}(\overline U)$. 
\item[(iii)]   $\omega_{f;U}(r)$ is $o(r^\alpha)$.
\item[(iv)] $\displaystyle \lim_{r\to 0}  [f]_{0,\alpha;U|r} =0$.
\end{itemize}
 \end{proposition}  
 
 \begin{proof}
 
 {(i)} $\Rightarrow$ {(iv)}:
Assuming $f\in c^{0,\alpha}(\overline U)$, let $f_k\in C^1 (\overline U)$ be a  sequence of functions converging to $f$ in 
 $C^{0,\alpha}(\overline U)$ as $k\to \infty$. We estimate: 
 $$
[f]_{0,\alpha|r}   \le  [f-f_k]_{0,\alpha|r} +  
[f_k]_{0,\alpha|r} \le  [f-f_k]_{0,\alpha}  + c(U)  \|\nabla f_k\|_0 r^{1-\alpha},
$$  and a straightforward argument by the auxiliary argument $k$ implies the required convergence.

{(iv)} $\Rightarrow$ {(iii)}:
This is straightforward since for all $r>0$:
$$
\frac{1}{r^\alpha} \omega_{f;U}(r) \le [f]_{0,\alpha;U|r}.
$$

{(iii)} $\Rightarrow$ {(ii)}:  By the assumption $f$ is uniformly continuous in $U$.  We extend $f$ to $\tilde f: \R^n\to \R$ by applying Proposition \ref{extension}, obtaining that $\omega_{\tilde f;\R^n}(r) \le C \omega_{f;U} (r)$ is $o(r^\alpha)$.
 We  conclude that 
$$
[\tilde f]_{0,\alpha;\R^n|r} = \sup_{0 \le s \le r}   \frac{1}{s^\alpha} \omega_{\tilde f;\R^n} (s)
$$ is $o(1)$ as a function of $r$. Consider any bounded weakly Lipschitz domain $V\subset \R^n$ and note that  $\tilde f|_V \in C^{0,\alpha}(\overline V)$. Applying Lemma \ref{litmolapp} to $V$ implies that  $\tilde f_\e|_V \to \tilde f|_V$ in $C^{0,\alpha}(\overline V)$ as $\e \to 0$ and  thus $\tilde f|_V \in c^{0,\alpha} (\overline V)$. Both conclusions in (ii) follow.   

{(ii)} $\Rightarrow$ {(i)}: This statement trivially follows   from Definition \ref{deflit}.

\end{proof}
  
\begin{corollary}\label{be>al}
Let $0<\alpha<\beta<1$ and let $U\subset \R^n$ be a bounded weakly Lipschitz domain. Then $c^{1,\beta}(\overline U) \subsetneq  C^{1,\beta} (\overline U) \subsetneq c^{1,\alpha}(\overline U)$.
\end{corollary}
  
 \begin{proof}
The first inclusion is trivial. If $f\in C^{1,\beta} (\overline U)$, then $\omega_{f;U}(r)$ is $O(r^\beta)$ and so the second inclusion follows by Proposition \ref{ltl0}.
 \end{proof}

\subsection{A distributional product and a criteria for H\"older continuity}
  
The following  two  propositions are  practically  known by the community at large. Similar or more general statements in the same spirit have appeared in various contexts, e.g.\@ in the discussion of the Young integral \cite{You} or of the paraproducts \cite[Chapter 4]{RuSi}.  For examples, see \cite[Theorem 2.52]{BCD}  and \cite[Theorem 13.16]{FH}, or compare with  \cite[Theorem 22]{Mar}. Here, we have formulated and proved rather accessible straightforward versions for the little H\"older spaces which are more adapted to our needs. 
 
 \begin{definition}\label{interpol}  Let $U\subset\R^n$ be a bounded weakly Lipschitz domain. We say that a distribution $T\in \mathcal{D}'(U,\R^n)$ satisfies  the  $\alpha$-interpolation property on  $W^{1,1}_0(U)$ whenever  for $C>0$
 $$
\forall \psi \in C^\infty_c (U)  \quad  \forall j\in \{1,\cdots, n\}\quad |T_j [\psi]| \le  C  \|\psi\|^{\alpha}_{L^1}\|\partial_j \psi\|^{1-\alpha} _{L^1}.
$$ We say that $T$ satisfies the fine $\alpha$-interpolation property on  $W^{1,1}_0(U)$ whenever  for $C>0$
 \begin{equation}\label{est2}
\begin{array}{ll}
 \forall \sigma>0 \quad \exists \delta= \delta(\sigma)>0  \quad \forall \psi \in C^\infty_c (U) \quad \forall j\in \{1,\cdots, n\} & \\ & 
 \hspace{-2.2in} \|\psi\|_{L^1} \le \delta \|\partial_j \psi\|_{L^1} \implies \displaystyle
  |T_j [\psi]| \le  C  \|\psi\|^{\alpha}_{L^1}\|\partial_j \psi\|^{1-\alpha} _{L^1} \sigma,
\end{array}
\end{equation}
with $\delta(1)= +\infty $ 
 \end{definition}

\begin{proposition}\label{lem1}
Let  $U\subset \R^n$ be a bounded weakly Lipschitz domain and $V \Subset U$ be an open set compactly contained in $U$.   
Let  $0 < \alpha < 1 $, $f \in C^{0,\alpha} (\overline U)$ and $h\in c^{0,\alpha}(\overline U)$. 
If  $\alpha >1/2$, then as $\e \to 0$,  for  all $j=1,\cdots n$ the sequence $f_\e \partial_j h_\e$ converges  in the sense of distributions 
to a distribution on $V$ which we denote by 
$f\partial_j h$. More precisely, the convergence is in the dual of $W^{1,1}_0(V)$, i.e.\@ if $\psi \in W^{1,1}_0(V)$ 
$$
f \partial_j h [\psi]: = \lim_{\e \to 0} \int_V f_\e \partial_j h_\e \psi(x) dx 
$$ exists. Moreover for all $\psi\in C^\infty_c(V)$
\begin{equation}\label{est1}
\Big |\int_V f_\e \partial_j h_\e  \psi(x) dx  - f\partial_j h[\psi] \Big | \le 
 \|f\|_{0,\alpha}  [h]_{0,\alpha} \Big (o(\e^{2\alpha -1})   \|\psi\|_{L^1}   +   o(\e^\alpha) \|\partial_j \psi\|_{L^1}  \Big ),
\end{equation}  and  $T= (f\partial_1 h, \dots, f\partial_n h)$  satisfies the fine $\alpha$-interpolation property \eqref{est2} on $W^{1,1}_0(U)$ with $C= c(U,V)   \|f\|_{0,\alpha} [h]_{0,\alpha}$.
\end{proposition}
\begin{remark}
If $h$ is merely in $C^{0,\alpha}(\overline U)$, a similar result holds true  with the following estimates for all $\psi \in W^{1,1}_0(U)$:

$$
\Big |\int_V f_\e \partial_j h_\e  \psi(x) dx  - f\partial_j h[\psi] \Big | \le 
c  \|f\|_{0,\alpha}  [h]_{0,\alpha} \Big ( \e^{2\alpha -1}   \|\psi\|_{L^1}   +   \e^\alpha \|\partial_j \psi\|_{L^1}  \Big )
$$   and
$$
   \quad 
  |f \partial_j h [\psi]| \le  C  \|\psi\|^{\alpha}_{L^1}\|\partial_j \psi\|^{1-\alpha} _{L^1},
$$ i.e.\@ $T= (f\partial_1 h, \dots, f\partial_n h)$  satisfies the  $\alpha$-interpolation property on $W^{1,1}_0(U)$.
\end{remark}
\begin{proof}
Let $\psi \in W^{1,1}_0(V)$, and  for $\e< {\rm dist}( \overline V, \partial U)$ we define:
$$
a^{(\e)} := \int_{V} f_\e \partial_j h_\e  \psi(x) dx. 
$$ We first prove that $a^{(\e)}$ is a Cauchy sequence.  For any $0<\e'<\e$ we have
 \begin{equation}\label{cauchy} 
\begin{aligned}
|a^{(\e)} - a^{(\e')}| &  = \Big |  \int_{V} (f_\e \partial_j h_\e - f_{\e'} \partial_j h_{\e'} ) \psi \Big | \\ & 
 \le \Big  |\int_{V} f_\e (\partial_j h_\e - \partial_j h_{\e'} ) \psi  \Big |  +   \Big   |\int_{V} (f_\e - f_{\e'}) \partial_j h_{\e'} \psi  \Big | 
 \\ & =   \Big |- \int_{V} \partial_j f_\e (h_\e - h_{\e'} ) \psi - \int_{V} f_\e (h_\e - h_{\e'} ) \partial_j \psi  \Big | 
 +   \Big  |\int_{V} (f_\e - f_{\e'}) \partial_j h_{\e'} \psi \Big | 
 \\ & \le   \Big | \int_{V} \partial_j f_\e (h_\e - h_{\e'} ) \psi \Big | + \Big |\int_{V} f_\e (h_\e - h_{\e'} ) \partial_j \psi \Big | +  
 \Big |\int_{V} (f_\e - f_{\e'}) \partial_j h_{\e'} \psi \Big | 
 \\ & = I_1 (\e, \e') + I_2 (\e, \e') + I_3(\e, \e')\\
 \end{aligned} 
\end{equation} 
We estimate, using $\e'<\e$: 
$$ 
\begin{aligned} 
I_1 (\e, \e') 
& = \Big |\int_{V} \psi(x) \Big ( \int_{\R^n}  (f(x-y) \partial_j \phi_\e(y)) dy \Big)  \Big (\int_{\R^n} (h(x-z) (\phi_\e(z) - \phi_{\e'}(z)) dz \Big )  dx  \Big | 
\\ &  = \Big |\int_{V} \psi(x)  \int_{\R^n}  (f(x-y) -f(x)) \partial_j \phi_\e(y) dy  \int_{\R^n} (h(x-z) -h(x) ) (\phi_\e(z) - \phi_{\e'}(z)) dz \, dx  \Big | 
\\ & \le  \|\psi\|_{L^1}   (c[f]_{0,\alpha}  \e^{\alpha-1} )(c[h]_{0,\alpha|\e} \e^{\alpha})  \le \|\psi\|_{L^1}   (c[f]_{0,\alpha}  \e^{\alpha-1} )([h]_{0,\alpha} o(\e^{\alpha})) \\ &  \le \|\psi\|_{L^1}  [h]_{0,\alpha} [f]_{0,\alpha} o(\e^{2\alpha-1}),
\end{aligned}   
$$ where we used Lemma \ref{estimcvls}-(i) and Proposition  \ref{ltl0} in the third line.  For $I_3(\e, \e')$, we note that $\e \not\le \e'$ and hence the same estimate as for $I_1$ is not obtained, but through the same calculations as for $I_1$, we obtain
$$
I_3 (\e, \e') \le    \|\psi\|_{L^1} [h]_{0,\alpha} [f]_{0,\alpha} {\e'}^{-1}  o({\e'}^\alpha)    \e^\alpha. 
$$
In the same manner, we can also obtain an estimate for $I_2$:
$$
\begin{aligned}
 I_2 (\e, \e') & =  \Big |\int_{V} f_\e (h_\e - h_{\e'} ) \partial_j \psi \Big | 
 \\ & \le  \|\partial_j \psi\|_{L^1}  \|f\|_{0}  \sup_{x\in {V}} \Big  |\int_{\R^n} h(x-y) (\phi_\e(y)  - \phi_{\e'}(y)) dy \Big |  
 \\ & =  \|\partial_j \psi\|_{L^1}  \|f\|_{0}    \sup_{x\in {V}}  \Big  |\int_{\R^n} (h(x-y) - h(x)) (\phi_\e(y)  - \phi_{\e'}(y)) dy \Big | 
 \\ & \le   \|\partial_j \psi\|_{L^1} \|f\|_{0}   [h]_{0,\alpha} o( \e^{\alpha}). 
\end{aligned} 
$$  
Putting the three estimates together in view of \eqref{cauchy}, we obtain a first crude estimate:
\begin{equation}\label{crude}
 |a^{(\e)} - a^{(\e')}|   \le  [h]_{0,\alpha} \Big ( \Big [ o(\e^{2\alpha-1})  + {\e'}^{-1}  o({\e'}^\alpha)    \e^\alpha\Big ]  [f]_{0, \alpha}   \|\psi\|_{L^1}   + o( \e^{\alpha})   \|f\|_{0}  \|\partial_j \psi\|_{L^1} \Big ).
\end{equation} This estimate is not enough to establish the Cauchy property of the sequence $a^{(\e)}$. Therefore  we  proceed as follows. 
If $\e'/\e \ge 1/2 $,  we set $\rho= \e'/\e$, $m=1$. Otherwise there exists $m \in \mathbb N$, depending on $\e'/\e<1$, such that $1/4 \le \rho=  ({\e'}/{\e})^{\frac 1m} <1/2$. We now write 
$$
 |a^{(\e)} - a^{(\e')}|   \le \sum_{k=1}^m  |a^{(\rho^{{k-1}}\e)} - a^{(\rho^{k}\e)}|. 
$$ We apply \eqref{crude} successively to $ 0 <\rho^{k}\e < \rho^{k-1}\e$ (as the new $\e'$ and $\e$) to obtain
$$
\begin{aligned}
 |a^{(\e)} - a^{(\e')}|   & \le    [h]_{0,\alpha} \Big (o(\e^{2\alpha-1})    [f]_{0, \alpha}  \|\psi\|_{L^1} (1+ \rho^{\alpha-1}  )  \sum_{k=1}^m  \rho^{{(k-1)}(2\alpha -1)}  \\ & \hspace{0.5in} + o( \e^{\alpha})  \|f\|_{0}   \|\partial_j \psi\|_{L^1} \sum_{k=1}^m  \rho^{(k-1)\alpha}   \Big ) 
 \\ &    \le      (\rho^{\alpha-1}  + 1 )  
 \frac{1- \rho^{(2\alpha-1)m}}{1- \rho^{2\alpha-1}}o(\e^{2\alpha-1})  [h]_{0,\alpha}    [f]_{0, \alpha}   \|\psi\|_{L^1} 
  \\ &   +  \frac{1- \rho^{\alpha m}}{1- \rho^{\alpha}}o( \e^{\alpha})   [h]_{0,\alpha}   \|f\|_{0} \|\partial_j \psi\|_{L^1}    
 \\ &  \le    o(\e^{2\alpha-1})  [h]_{0,\alpha}   [f]_{0, \alpha}  \|\psi\|_{L^1}     + o( \e^{\alpha}) [h]_{0,\alpha}    \|f\|_{0}   \|\partial_j \psi\|_{L^1},
\end{aligned} 
$$ since by our choice  either $1/2 \le \rho <1$ and $m=1$, or $1/4\le \rho <1/2$, independent of $\e', \e$. This implies that $a^{(\e)}$ is Cauchy sequence. Hence the limit  
$$
f \partial_j h [\psi]:= a = \lim_{\e \to 0}   a^{(\e)},
$$ exists.  
  
Now, we estimate for a fixed $\e>0$, as $\e'\to 0 $:
$$
\Big |\int_{V} f_\e \partial_j h_\e \psi(x) dx  - f\partial_j h[\psi] \Big | = |a^{(\e)} -a| \le 
\|\psi\|_{L^1} [f]_{0,\alpha} [h]_{0,\alpha}  o(\e^{2\alpha -1}) +  \|\partial_j \psi\|_{L^1}  \|f\|_{0} [h]_{0,\alpha} o(\e^\alpha),
$$ which establishes \eqref{est1}. To obtain \eqref{est2}, applying Lemma \ref{estimcvls}-(ii) and Proposition  \ref{ltl0} to $\partial_j h_\e$ we observe that:
$$
\begin{aligned} 
|f\partial_j h [\psi]| & \le |a^{(\e)}| + |a^{(\e)} - a| \le |a^{(\e)}| +  [f]_{0,\alpha} [h]_{0,\alpha} \|\psi\|_{L^1} o( \e^{2\alpha -1} )  + 
    \|f\|_{0} [h]_{0,\alpha} \|\partial_j \psi\|_{L^1}  o(\e^\alpha)    
\\ & \le    \|f\|_0   [h]_{0,\alpha}   \|\psi\|_{L^1}  o(\e^{\alpha-1})  + 
  [f]_{0,\alpha} [h]_{0,\alpha} \|\psi\|_{L^1}  o(  \e^{2\alpha -1}) +    \|f\|_{0} [h]_{0,\alpha}  \|\partial_j \psi\|_{L^1}  o(  \e^\alpha)  
\\ & \le   \|f\|_{0,\alpha} [h]_{0,\alpha}  \Big (\|\psi\|_{L^1} o(\e^{\alpha-1}) +   \|\partial_j \psi\|_{L^1} o(\e^{\alpha}) \Big )  
\end{aligned}   
$$ In other words,  
\begin{equation}\label{epsilon} 
\hspace{-0.05in}\begin{array}{ll}
\mbox{for any}\,\, \sigma >0,  \mbox{there exists}\,\, \delta_0= \delta_0(\sigma)>0\,\, \mbox{such that} & \\  
\e \le \min\{\delta_0, {\rm dist} (\overline V, \partial U)\} \implies |f\partial_j h [\psi]| \le  c \|f\|_{0,\alpha} [h]_{0,\alpha} \Big (   \|\psi\|_{L^1} \e^{\alpha-1}+   \|\partial_j \psi\|_{L^1} \e^\alpha \Big )   \sigma, &
\end{array}
\end{equation} with $\delta_0 (1) = +\infty$. 
 
We conclude the proof of \eqref{est2}:   Given, $\sigma >0$, we let  
$$
\displaystyle \delta = \Big (1+\frac{{\rm diam} (V)}{ {\rm dist} (\overline V, \partial U)}\Big )  \delta_0.
$$
Assume $\|\psi\|_{L^1} \le \delta \|\partial_j \psi\|_{L^1}$ for a given nonzero $\psi$, then  letting
 $$ 
 \e = \Big (1+\frac{{\rm diam} (V)}{ {\rm dist} (\overline V, \partial U)}\Big )^{-1} \frac{\|\psi\|_{L^1}}{\|\partial_j \psi\|_{L^1}},  
$$ we obtain $\e <  {\rm dist} (\overline V, \partial U)$ by the Poincar\'e inequality 
$
\|\psi\|_{L^1} \le {\rm diam}(V) \|\partial_j \psi\|_{L^1}
$ on $V$.  On the other hand we also have $\e \le \delta_0$ and therefore applying \eqref{epsilon} we obtain  as required:
$$
|f\partial_j h [\psi]| \le c   \Big (1+\frac{ {\rm diam} (V)}{{\rm dist} (\overline V, \partial U)}\Big )^{1-\alpha} \|f\|_{0,\alpha} [h]_{0,\alpha}  \|\psi\|^\alpha_{L^1}{\|\partial_j \psi\|^{1-\alpha}_{L^1}} \sigma.
$$\end{proof}  
 
\begin{definition}
We say $U \subset\R^n$ is an open coordinate rectangular box when it is a rectangular cuboid with edges parallel to the coordinate axes, i.e.\@
$$\displaystyle U=\prod_{k=1}^n I_k$$ for $I_k$ an open interval in $\R$.   
 \end{definition}

\begin{proposition}\label{C-alpha-control} Assume  that $U \subset\R^n$ is an open  coordinate rectangular box and that for $f\in L^1(U)$, $T=\nabla f$ satisfies the fine  $\alpha$-interpolation property on $W^{1,1}_0(U)$ according to Definition \ref{interpol} with $C=1$. 
Then 
$$
[f]_{0,\alpha;U|\delta(\sigma)} \le 2^{n+1} \sigma.
$$ In particular, by Proposition \ref{ltl0}, $f\in c^{0,\alpha} (\overline U)$ and  
$
[f]_{0,\alpha;U} \le 2^{n+1}.   
$ 
 
\end{proposition}

\begin{proof} 

The proof of the proposition follows classical ideas relating the decay of mean oscillations to pointwise behavior of functions, as pioneered by Morrey and Campanato, see for instance \cite{C63}. We prove the statement by induction over the dimension. 

First let $n=1$ and $U \subset \R$ and be an open interval and let $I\subset U$ be any open subinterval.  
Note that for any $\phi\in C^\infty_c (I)$,  we can construct a sequence $\tilde \phi_k \in C^\infty_c(I)$ such that  $\|\tilde \phi_k\|_{0}\le 1+ 2 \|\phi\|_{0}$, and $\tilde \phi_k$ converges
strongly in $L^1$  to $ \phi - \fint_I \phi$. Now fix $\varphi \in C^\infty_c(I)$ with $\fint_I \varphi =1$ and let 
$$
\phi_k := \tilde \phi_k - \Big (\fint_I \tilde \phi_k \Big ) \varphi \xrightarrow{L^1}  \phi - \fint_I \phi.
$$ By the dominated convergence theorem we have
$$
\lim_{k\to \infty}  \int_I f\phi_k =  \int_I f (\phi -\fint_I \phi). 
$$ Let $\psi_k (t) = \int_{-\infty}^t \phi_k(s) ds$ and note that 
$\psi_k \in C^\infty_c(I)$. We have therefore for all $\phi\in C^\infty_c(I)$ and $\sigma >0$, provided $|I|\le \delta (\sigma)$: 
$$
\begin{aligned}
\Big |\int_I (f - \fint_I f) \phi \Big | & = \Big | \int_I f (\phi -\fint_I \phi) \Big|  = 
\lim_{k\to \infty} \Big| \int_I f\phi_k \Big|  =   \lim_{k\to \infty} \Big | \int_I f\psi'_k\Big |  =  \lim_{k\to \infty} \Big |T [\psi_k]\Big | \\ &  \le 
 \liminf_{k\to \infty}  \|\psi_k\|_{L^1}^{\alpha}    \|\psi'_k\|^{1-\alpha} _{L^1} \sigma
 \le \liminf_{k\to \infty}    \|\psi'_k\|_{L^1}    |I|^\alpha \sigma  \\ & = 
   \liminf_{k\to \infty} \|\phi_k - \fint_I \phi_k\|_{L^1}  |I|^\alpha \sigma =  \|\phi - \fint_I \phi\|_{L^1}  |I|^\alpha\sigma
     \le 2\|\phi\|_{L^1}  |I|^\alpha \sigma,
\end{aligned} 
$$ where we used the Poincar\'e inequality on $I$.   Since $(L^1)' = L^\infty$, and $C^\infty_c(I)$ is dense in $L^1(I)$, we conclude that
\begin{equation}\label{bound}
\forall \sigma>0 \quad \quad   |I| \le \delta(\sigma) \implies \|f- \fint_I f\|_{L^\infty(I)} \le  2 |I|^\alpha \sigma.
 \end{equation} Now, let $y\in U$, and let $ d= {\rm dist}(y, \partial U)$, and for  $z\in (y-d/2, y+d/2)$ and $r<d/2$ define
 $$
 h_r(z):= \fint_{z-r}^{z+r} f(x) dx. 
 $$  $h_r$ is continuous in $z$ and converges a.e. to $f$ on  $(y-d/2, y+d/2)$ as $r\to 0$. On the other hand, for all $r'<r<d/2$:
  $$
\begin{aligned}
\displaystyle \Big | h_r (z) -  h_{r'}(z)\Big | & \le \fint_{(z-r, z+r) \times (z-r', z+r') } |f(x) - f(x')| dx dx'  \le  4|(z-r, z+r)|^\alpha = 2^{\alpha+2} r^\alpha, 
\end{aligned}
$$ where we applied the bound \eqref{bound} to $I=(z-r, z+r)$, $\sigma=1$, $\delta(1)=+\infty$.  As a consequence, $h_r$ is a Cauchy sequence in the uniform norm and $f$ is  continuous 
as the uniform limit of the $h_r$ for $r \to 0$.  Now,  applying once again \eqref{bound} to $I= (x,y) \subset U$ we obtain that
$$
\forall \sigma>0 \quad \forall x,y \in U\quad |x-y| \le \delta(\sigma) \implies  |f(x) - f(y)| \le 4 |x-y|^\alpha \sigma,
$$  which implies $[f]_{0,\alpha;U|\delta(\sigma)} \le 4 \sigma$, as required.  

Now, assume that $n>1$ and that the statement is true for $n-1$. Let $Q\subset U $ be any coordinate rectangular box, i.e.\@ 
$$
Q:=\displaystyle \prod_{k=1}^n I_k   \subset U,
$$ with the open intervals $I_k\subset \R$. 
Let 
$\widehat{x}:= (x_1, \cdots, x_{n-1})$  and note that for all $h\in L^1(Q)$, 
$$
{\rm for \,\, a.e.} \,\, \widehat{x}  \in Q^{n-1} := \prod_{j=1}^{n-1}  I_k ,  \,\, 
\hat h(\widehat{x})  := \fint_{I_n} h(x) dx_n
$$ is well defined and belongs to $L^1(Q)$.  We claim that: 
\begin{equation}\label{boundn}
\forall \sigma>0\quad |I_n| \le \delta(\sigma) \implies \esssup_{x\in Q} \Big |  f(x) -  \hat f(\widehat{x}) \Big |    \le 2 |I_n| ^\alpha \sigma.
\end{equation}

Let us first show that the conclusion holds assuming the claim \eqref{boundn} is true for all  coordinate boxes 
$Q\subset U$.  For $a,b$ the extremities of $I_n$,  choose 
$0 \le \theta_k \le 1$ in $C^\infty_c (I_n)$ such that  $\theta_k \equiv 1$ on $(a+1/k, b-1/k)$.   Given  $\sigma>0$,  for $1 \le j \le n-1$, 
and $\psi \in C^\infty_c (Q^{n-1})$,  assuming that  $\|\psi\|_{L^1(Q^{n-1})} \le \delta(\sigma) \|\partial_j \psi\|_{L^1(Q^{n-1})}$,  we have by our main assumption on $f$:
$$
\begin{aligned}
\Big | \int_{Q^{n-1}} \hat f \partial_j \psi \Big |  & = \frac{1}{|I_n|} \Big | \int_Q f(x) \partial_j \psi(\widehat x) dx \Big |  
= \lim_{k \to \infty}  \frac{1}{|I_n|} \Big | \int_Q f(x) \theta_k (x_n)\partial_j \psi(\widehat{x}) dx \Big | \\ & 
=  \lim_{k \to \infty}  \frac{1}{|I_n|} \Big | \int_Q f  \partial_j (\theta_k \psi) \Big | 
\le  \lim_{k \to \infty}   \frac{1}{|I_n|} \|\theta_k \psi\|^\alpha_{L^1(Q)} \|\partial_j(\theta_k  \psi)\|^{1-\alpha}_{L^1(Q)}  \sigma
\\ & =  \lim_{k \to \infty}   \frac{1}{|I_n|} \|\theta_k \psi\|^\alpha_{L^1(Q)} \| \theta_k \partial_j\psi \|^{1-\alpha}_{L^1(Q)} \sigma
=  \frac{1}{|I_n|}  \|\psi\|^\alpha_{L^1(Q)} \|\partial_j\psi\|^{1-\alpha}_{L^1(Q)}  \sigma
\\ & = \|\psi\|^\alpha_{L^1(Q^{n-1})} \|\partial_j \psi\|^{1-\alpha}_{L^1(Q^{n-1})}  \sigma, 
\end{aligned} 
$$  since $\theta_k \psi \in C^\infty_c(Q)$ and  
$$\|\theta_k \psi\|_{L^1(Q)}  =\Big (\int_a^b \theta_k \Big )  \|\psi\|_{L^1(Q^{n-1})}
\le   \delta(\sigma) \Big (\int_a^b \theta_k \Big ) \|\partial_j \psi\|_{L^1(Q^{n-1})} = \delta(\sigma)  \|\partial_j (\theta_k  \psi)\|_{L^1(Q)}.
$$ Applying the induction assumption, we deduce that
\begin{equation}\label{step}
[\hat f]_{0,\alpha; Q^{n-1}|\delta(\sigma)} \le 2^{n}.
\end{equation} 

To  prove that $f$ is continuous, fix $y\in U$ and consider a box 
$Q_d = Q^{n-1} \times  (y_n-d, y_n+d) \subset U$ containing $y$. 
For all $z\in Q^{n-1} \times (y_n -d/2, y_n+ d/2)$, and $r<d/2$ we  define   
$$
h_r(z):= \fint_{z_n - r}^{z_n+ r} f(\widehat{z}, s) ds  
$$ Note that, if $r$ is fixed, the vertical averages of $f$ are continuous in $\widehat{z}$ as established in \eqref{step}, 
and so $h_r$ is continuous in $z$.  Applying \eqref{boundn} with $\sigma=1$ to   $Q=Q^{n-1} \times (z_n -r, z_n+ r)$ we have for $r'<r<d/2$: 
$$
\begin{aligned}
\displaystyle \Big | h_r (z) -  h_{r'}(z)\Big | & \le \fint_{(z-r, z+r) \times (z-r', z+r') } |f(\widehat{z},s) - f(\widehat{z},s')| ds ds'  \le  4|(z-r, z+r)|^\alpha = 2^{\alpha+2} r^\alpha, 
\end{aligned}
$$ which implies that $h_r$ locally uniformly converge to their limit, which happens to be $f$. Hence $f$ is continuous in 
$U$. 

To obtain a  H\"older estimate, let $x,y\in  U$. First, assume $x_j=y_j$  for some $1 \le j\le n$.  
We re-arrange the coordinates so that  $j=n$ and note that for any sequence of coordinate rectangular boxes 
$Q_k$ of height $1/k$, containing $x$ and $y$, such that $x_n=y_n$ is the midpoint of $I_{n,k}$, we have 
by \eqref{step}
 $$
\forall \sigma>0\quad |x-y| \le \delta(\sigma) \implies |f(x) - f(y)| = \lim_{k\to \infty} |\hat f_{Q_k}(\widehat{x})- \hat f_{Q_k}(\widehat{y})| 
\le 2^{n} |\widehat{x} - \widehat{y}|^\alpha\sigma  \le 2^n |x-y|^\alpha \sigma. 
$$ If, on the other hand, $x_j\neq y_j$ for all $j$, we can  choose a coordinate rectangular box $Q\subset U$ which has $x$ and $y$ as its opposite vertices on the largest diameter. We obtain by applying   \eqref{boundn}  and \eqref{step}
to the now continuous $f$, provided $|x-y| \le \delta(\sigma)$: 
$$
\begin{aligned}
|f(x) - f (y) | & \le |f(x) - \hat f (\widehat{x})| + | \hat f (\widehat{x}) - \hat f (\widehat{y}) |  + |f(y) - \hat f (\widehat{y}) |   \\ & \le  
4 |x_n-y_n|^\alpha \sigma +  2^{n}  |\widehat{x} -\widehat{y}|^\alpha\sigma   \le (4+ 2^n)  |x-y|^\alpha   \sigma \\ & \le   2^{n+1} |x-y|^\alpha \sigma.
\end{aligned} 
$$ In both cases we can conclude with the desired bound, i.e.\@ $[f]_{0,\alpha; U |\delta(\sigma)} \le 2^{n+1}\sigma$.  Note that the sequence $Q_k$ in the first case and the box $Q$ in the second case exist since  $U$ is assumed to be a coordinate box itself.

\medskip 

It remains to prove the claim \eqref{boundn}. Given $\sigma>0$ assume that $|I_n|\le \delta(\sigma)$ as required in \eqref{boundn}.  For any $\phi\in C^\infty_c (Q)$,   consider  a sequence $\tilde \phi_k \in C^\infty_c(Q)$, $\|\tilde \phi_k\|_0 \le 1 + 2\|\phi\|_0$,  converging strongly in $L^1$  to 
$ \phi - \hat \phi(\widehat{x})$.  As a consequence,
$$
\lim_{k\to \infty} \int_{I_n} \tilde \phi_k(x) dx_n  =0 \quad \mbox{in}\,\, L^1(Q^{n-1}). 
$$   Choose a $\varphi \in C^\infty_c(I_n)$ such that $\fint_{I_n} \varphi =1$ and define
$$
\phi_k (x) = \tilde \phi_k(x)  -  \Big (\fint_{I_n} \tilde \phi_k (\widehat{x}, s ) ds \Big )\varphi (x_n).
$$ Then $\phi_k \in C^\infty_c(Q)$ converges in $L^1$ to $\phi - \hat \phi(\widehat{x})$ 
and $\hat \phi_k (\widehat{x}) = \fint_{I_n} \phi_k(x) dx_n =0$.  Letting $\psi_k (\widehat {x}, t) = \int_{-\infty}^t \phi_k(\widehat{x}, s) ds$ we obtain for all $\phi\in C^\infty_c (Q)$, $\psi_k \in C^\infty_c(Q)$. We hence obtain 
$$
\begin{aligned}
\Big |\int_Q (f -  \hat f (\widehat{x})) \phi \Big | & = \Big | \int_Q f (\phi - \hat \phi (\widehat{x}))\Big|  = 
\lim_{k\to \infty} \Big| \int_Q f\phi_k \Big|  =   \lim_{k\to \infty} \Big | \int_Q f \partial_n \psi_k \Big |  = \lim_{k\to \infty} \Big | T_n [\psi_k] \Big | \\ & 
\le  \liminf_{k\to \infty} \|\psi_k\|^\alpha _{L^1} \|\partial_n \psi_k\|^{1-\alpha}_{L^1}\sigma  
 \le \liminf_{k\to \infty}  |I_n|^\alpha  \|\partial_n \psi_k\|^\alpha _{L^1}  \|\partial_n \psi_k\|^{1-\alpha}_{L^1}  \sigma \\ & = 
   \lim_{k\to \infty} |I_n|^\alpha \|\phi_k - \fint_{I_n} \phi_k(x) dx_n \|_{L^1} \sigma = |I_n|^\alpha  
   \|\phi - \fint_{I_n} \phi (x) dx_n\|_{L^1(Q)}\sigma 
   \\ & \le 2 |I_n|^\alpha \|\phi\|_{L^1(Q)}\sigma,
\end{aligned} 
$$ where we used the Poincar\'e inequality 
$$
\|\psi_k\|_{L^1(Q)} \le |I_n| \|\partial _n \psi_k \|_{L^1(Q)} \le \delta(\sigma) \|\partial_n \psi_k \|_{L^1(Q)}
$$ on $I_n$.   Once again, using the fact that $L^\infty$ is the dual  of $L^1$, and the density of $C^\infty_c(Q)$ in $L^1(Q)$, 
we conclude with \eqref{boundn}. 
\end{proof}

 \section{A second fundamental form for $c^{1,\alpha}$ isometries} \label{2ndform}

 Let u$\in c^{1,\alpha}(\Omega, \R^3)$ be an immersion with $2/3 \le \alpha < 1$ and let $V \Subset \Omega$ be 
 any smooth  simply-connected domain compactly contained in $\Omega$, e.g.\@ an open disk. We intend to show the following three statements, which we will formulate more precisely later on.  We shall prove:

\begin{itemize}

\item[1.] A weak notion of the second fundamental form $A_{ij}= -\partial_i u \cdot \partial_j \n$ 
makes sense for the immersion $u$ on $V$.  $A=[A_{ij}]$ is symmetric.  This part can be carried out for $\alpha >1/2$.

\item[2.] If $u$ is isometric, then $A$ is curl free, and is equal to $\nabla^2 v$ for a scalar function $v\in c^{1,\alpha}(V)$. 
\item[3.] $\displaystyle \Det (D^2 v) = -\frac 12 \cc (\nabla v \otimes \nabla v)=0$ in $\mathcal D'(V)$. 
 \end{itemize}

 We choose an open smooth domain $U$ such that $V \Subset U \Subset \Omega$. Let $\phi \in C^\infty_c (B(0,1))$ be a nonnegative 
 function with $\int \phi =1$.   For  $\e < {\rm dist}(\overline V, \partial U)$ we define 
$$
\n_\e:= \n \ast \phi_\e
$$ 
to be the mollification of the unit normal $$\vec n = \frac{\partial_1 u\times \partial_2 u} {|\partial_1 u \times \partial_2 u|}$$ of immersion $u$.  Note that by assumption $C_0:= \|\partial_1 u \times \partial_2 u\|_{0;U}>0$  and $\vec n \in c^{0,\alpha} (\Omega,\R^3)$.   Let  $u_\e := u \ast \phi_\e : V\to \R^3$ be the standard  mollification of $u$ on $V$. If $\e$ is small enough $u_\e$ is an immersion since $\nabla u_\e$ converges uniformly to $\nabla u$ on $V$, and 
$$
|\partial_1 u_\e \times \partial_2 u_\e| \ge  \frac 12 C_0.
$$
We  define 
$$
 N^{\e}:= \frac {\partial_1 u_\e\times \partial_2 u_\e}{|\partial_1 u_\e\times \partial_2 u_\e|}
$$ 
to be the unit normal corresponding to $u_\e$. A direct application of Lemma \ref{estimcvls} yields:
\begin{equation}\label{estims}
\|\partial_i u_\e-  \partial_i u\|_{0;V}  \le   \|u\|_{1,\alpha} o(\e^\alpha), \,\,  \|u_\e\|_{2;V} \le  \|u\|_{1,\alpha} o(\e^{\alpha-1});  \quad \|\vec n_\e - \vec n\|_{0;V} \le  \|u\|^2_{1,\alpha}  o(\e^\alpha). 
\end{equation}  We also  note that
\begin{equation}\label{estimN}
\| N^{\e} - \n\|_{0;V} \le  \frac{1}{C_0} \|u\|^2_{1,\alpha}o(\e^\alpha) ,\,\,   \quad  \|\nabla N^\e\|_{0;V} \le  \frac{1}{C_0} \|u\|^2_{1,\alpha} o(\e^{\alpha-1}).
\end{equation}
To see the first estimate  we observe that 
$$
\begin{aligned} 
\| N^{\e} - \n\|_{0;V} & \le \Big \| \frac {\partial_1 u_\e\times \partial_2 u_\e}{|\partial_1 u_\e\times \partial_2 u_\e|} - \frac{\partial_1 u \times \partial_2 u}{|\partial_1 u \times \partial_2 u |} \Big \|_{0;V}
\\ & \le  \frac{2}{C_0^2} \Big \| |\partial_1 u \times \partial_2 u | (\partial_1 u_\e\times \partial_2 u_\e) - |\partial_1 u_\e\times \partial_2 u_\e| (\partial_1 u \times \partial_2 u) \Big \|_{0;V}
\\ & \le  \frac{2}{C_0}  \Big ( \Big \| \partial_1 u_\e\times \partial_2 u_\e - \partial_1 u \times \partial_2 u \Big \|_{0;V} + \Big \| |\partial_1 u_\e\times \partial_2 u_\e|  -|\partial_1 u \times \partial_2 u|  \Big \|_{0;V} \Big ) \\ & 
\le  \frac{4}{C_0} \Big \| \partial_1 u_\e\times \partial_2 u_\e - \partial_1 u \times \partial_2 u \Big \|_{0;V} \le  \frac{1}{C_0}  \|u\|^2_{1, \alpha} o(\e^\alpha),
\end{aligned}  
$$ where we used \eqref{estims}. The second estimate also follows from \eqref{estims} and a direct calculation.

We finally denote by the symmetric matrix field $A^{\e} :=[A^\e_{ij}]$ the second fundamental form associated with the immersion $u_\e$ in the local coordinates, i.e.\@
\begin{equation}\label{Ae}
A^{\e}_{ij}:= \partial_{ij} u_\e \cdot  N^{\e}=  - \partial_i u_\e \cdot \partial_j  N^{\e},
\end{equation}  which by \eqref{estims} satisfies the obvious bound  
\begin{equation}\label{estimA}
\|A^\e\|_{0;V} \le [u_\e]_{2;V} \le \|u\|_{1, \alpha}o(\e^{\alpha-1}).
\end{equation} 

\subsection{Definition of the second fundamental form $A$}

We are ready now to prove our main claims about the existence  of a second fundamental form for the immersion $u$ and its properties. In what follows the constant $C=c(U, V, \|u\|_{1,\alpha ;U}, C_0)$ might change but it has  a universal upper bound only depending on the stated quantities. 

\begin{proposition}\label{2nd}
If  $\alpha > 1/2$, then $A^{\e}_{ij}$ converges in the sense of distributions to a distribution $A_{ij} \in \mathcal D'(V)$. More precisely,  the convergence is in the dual of $W^{1,1}_0(V)$, and we have:  
$$
\forall \psi\in  C^\infty_c(V) \quad  |(A^{\e}_{ij} - A_{ij})[\psi]| \le C \Big (o(\e^{2\alpha-1}) \|\psi\|_{L^1} + o(\e^\alpha) \|\partial_j \psi\|_{L^1} \Big).
$$ Moreover, for each $i$, the distribution $T=(A_{ij})_{j=1}^n$ satisfies the fine $\alpha$-interpolation  property on $W^{1,1}_0(V)$ according to Definition \ref{interpol}. 
\end{proposition}

\begin{proof}
Applying Proposition \ref{lem1} to the components $f:= \partial_i u^m$, $h:= \vec n^m$, for $m\in \{1,2,3\}$, we deduce that  for a distribution $A_{ij} \in \mathcal{D'}(V)$,
$$
B^{\e}_{ij} := - \partial_i u_\e \cdot \partial_j \vec n_\e \to  A_{ij} := - \partial_i u \cdot \partial_j \vec n   \quad \mbox{in}\,\, \mathcal{D'}(V),
$$   and
$$
|B^{\e}_{ij}[\psi]  - A_{ij} [\psi]| \le C \Big ( o(\e^{2\alpha-1}) \|\psi\|_{L^1}  + o(\e^\alpha) \|\partial_j \psi\|_{L^1} \Big).
$$
 On the other hand, we have for all $\psi \in W^{1,1}_0(V)$:
$$
\begin{aligned}
|A^{\e}_{ij} [\psi] - A_{ij} [\psi]| &  \le |B^{\e}_{ij}[\psi]  - A_{ij} [\psi]| + \Big |\int_B \partial_i u_\e(x) \cdot  (\partial_j {\vec n}_\e (x) - \partial_j N^{\e} (x))\psi(x) dx \Big | 
\\ & \le C \Big ( o(\e^{2\alpha-1}) \|\psi\|_{L^1}  + o(\e^\alpha) \|\partial_j \psi\|_{L^1} \Big) \\ &  +  \Big |\int_V \partial_{ij} u_\e \cdot  ({\vec n}_\e  - N^{\e})\psi dx \Big | + \Big |\int_V \partial_i u_\e \cdot  ({\vec n}_\e - N^{\e}) \partial_j \psi \Big |  \\ & \le  C \Big ( o(\e^{2\alpha-1}) \|\psi\|_{L^1}  + o(\e^\alpha) \|\partial_j \psi\|_{L^1} \Big), 
\end{aligned} 
$$ where we used  Proposition \ref{lem1}, \eqref{estims} and \eqref{estimN}.
The stated fine $\alpha$-interpolation property follows as in Proposition \ref{lem1}. 
\end{proof}

\begin{definition}\label{2nddef}
For  $1/2 < \alpha <1$, we define the weak second fundamental form of an immersion $u\in c^{1,\alpha}(\Omega, \R^3)$ on $V$ as the distribution $$A: C^\infty_c (V) \to \R_{sym}^{2\times 2},   \quad A[\psi]:= [A_{ij}[\psi]]_{2\times 2}.$$  
\end{definition}
 
 \begin{remark}\label{whldmn}
 The definition of $A$  can be  uniquely extended to a distribution globally defined on $\Omega$, i.e.\@ belonging to $\mathcal D'(\Omega, \R_{sym}^{2\times 2})$, but we will not need this fact.
 \end{remark}
\subsection{$A$ satisfies Gauss-Codazzi-Mainardi equations (in a weak sense)}

From now on we assume that  the immersion $u$ is also isometric, i.e.\@  the pull back-metric  $g:= (\nabla u)^T (\nabla u)$ is the Euclidean metric $\E_2$ on $\Omega$. We consider  $\g^{\e}:= (\nabla u_\e)^T \nabla u_\e = u_\e^\ast \E_3$,  the pull-back metric induced by $u_\e$, where $\E_3$ is the Euclidean  metric of $\R^3$.  The  commutator estimate Lemma \ref{estimcvls}-(iii) implies
\bees
\|\g^{\e} - (g \ast \phi_\e)\|_{1;V} = \|(\nabla u_\e)^T \nabla u_\e - ((\nabla u)^T \nabla u)_\e \|_{1;V} \le \|u\|^2_{1,\alpha} o(\e^{2\alpha-1}).  
\eees  
So, in view of the  fact that $\E_2\ast \phi_\e = \E_2$ on $V$ we obtain:
\bee\label{gest}
\|\g^{\e} - \E_2\|_{1;V} \le  \|u\|^2_{1,\alpha}o( \e^{2\alpha-1}).  
\eee  Compare with \cite[Proposition 1]{CDS}, which is stated under more general settings and where
the mollifier needs to be symmetric.

\begin{proposition}\label{curl0} We have the uniform estimate 
\begin{equation}\label{curle0}
\|{\rm curl}\, A^{\e}\|_{0;V} \le C o(\e^{3\alpha-2}).
\end{equation} In particular, if $\alpha\ge 2/3$, then ${\rm curl} \, A =0$  in $\mathcal D'(V)$.
\end{proposition} 
\begin{proof}
We recall that the Christoffel symbols associated with a metric $g$ are given by
$$
\Gamma^i_{jk} (g) = \frac 12 g^{im} (\partial_k g_{jm} +   \partial_j g_{km} - \partial_m g_{jk}),
$$  where the Einstein summation convention is used.  Hence, in view of \eqref{gest}, we have
$$
\Gamma^{i,\e}_{kj} := \Gamma^i_{jk} (\g^{\e}) \to 0
$$ uniformly, if $\alpha\ge 1/2$, with the estimate
\begin{equation}\label{Cffl}
\|\Gamma^{i,\e}_{kj}\|_{0;V}\le C o(\e^{2\alpha-1}).
\end{equation}
Writing the Codazzi-Mainardi equations \cite[Equation (2.1.6)]{hanhong} for the immersion $u_\e$ we have
$$
\begin{aligned} 
\partial_2 A^{\e}_{11} -\partial_1 A^{\e}_{12}= A^{\e}_{11} \Gamma_{12}^{1, \e}
+A^{\e}_{12} (\Gamma_{12}^{2,\e}-\Gamma_{11}^{1,\e}) -A^{\e}_{22} \Gamma_{11}^{2,\e},\\	
\partial_2 A^{\e}_{12} -\partial_1 A^{\e}_{22}=	A^{\e}_{11}\Gamma_{22}^{1,\e}
+A^{\e}_{12} (\Gamma_{22}^{2,\e}-\Gamma_{12}^{1,\e})-A^{\e}_{22} \Gamma_{12}^{2,\e}.
\end{aligned} 
$$ 
Hence, by \eqref{estimA} and \eqref{Cffl},
\begin{equation*} 
\|\partial_2 A^{\e}_{i1} - \partial_1 A^{\e}_{i2}\|_{0;V} \le C o(\e^{\alpha-1}) o(\e^{2\alpha-1}), 
\end{equation*}  which  completes the proof.
\end{proof} 
 
 Before proceeding to the next step, we need to  establish a further property of the sequence  $A^\e$. Observe that if  $u \in C^2(\Omega, \R^3)$ were an isometric immersion, 
 then the determinant of its second fundamental form $A$ would vanish, since, in this particular case, 
 it would be equal to the Gaussian curvature of the Euclidean metric. In the case considered here, $u$ does not enjoy the sufficient regularity for $\det A$ to be defined. However,  we can show that for $\alpha\ge 1/2$, $\det A^\e$ converges in the sense of  distributions to $0$ as $\e\to 0$, that with a rate which will be crucial in our analysis. 
The following statement is  a slight variant of \cite[Proposition 7]{CDS} recast for our purposes:
\begin{proposition}\label{Adet0}
Let $A^\e$ be as defined in \eqref{Ae}.  Then for all $\psi \in C^\infty_c(V)$, $$\Big | \int_V (\det A^\e) \psi \Big | \le C  \|\psi\|_{W^{1,1}} o(\e^{2\alpha-1}). $$
\end{proposition} \begin{proof}
In view of the formula for the $(0,4)$-Riemann curvature tensor \cite[Equation (2.1.2)]{hanhong}
$$
R_{iljk} = g_{lm} (\partial_k \Gamma^m_{ij}  -\partial_j \Gamma^m_{ik}   + \Gamma^m_{ks} \Gamma^s_{ij} - \Gamma^m_{js} \Gamma^s_{ik}),  
$$ and by the Gauss equation \cite[Equation (2.1.7)]{hanhong}, \eqref{gest} and \eqref{Cffl} we have on $V$:
$$
\begin{aligned}
\det A^\e & = R_{1212}(\g^\e) = \g^\e_{1m}(\partial_1 \Gamma^{m,\e}_{22} - \partial_2 \Gamma^{m,\e}_{21} + \Gamma^{m,\e}_{1s} \Gamma^{s,\e}_{22}  
- \Gamma^{m,\e}_{2s} \Gamma^{s,\e}_{21})\\ & =
\partial_1 (\g^\e_{1m} \Gamma^{m,\e}_{22} ) - \partial_2 ( \g^\e_{1m} \Gamma^{m,\e}_{21})  + o(\e^{2\alpha-1})  
\\&  =  2 \partial_{12} \g^\e_{12} - \partial_{11} \g^\e_{22} - \partial_{22} \g^\e_{11} + o(\e^{2\alpha-1})  \\ & = - {\rm curl}^T {\rm curl}\,\, \g^\e  + o(\e^{2\alpha-1}).
   \end{aligned} 
$$ The conclusion  follows  by an integration by parts involving the first term and applying \eqref{gest}.  
\end{proof}

\subsection{$A$   as the Hessian of a scalar function}

\begin{proposition}\label{AHess}
Let $\alpha\ge 2/3$, $u\in c^{1,\alpha}(\Omega, \R^3) $ be an isometric immersion and $V$ be a smooth simply-connected domain compactly contained in $\Omega$. There exists $v\in c^{1,\alpha} (V) \cap W^{1,2}(V)$  such that $A$, the 2nd fundamental form of $u$ on $V$, equals $\nabla^2 v$ in the sense of distributions. Moreover
$$
\Det D^2 v  = -\frac 12 {\rm curl^T curl} (\nabla v \otimes \nabla v)=0 \quad \mbox{in} \,\, \mathcal D'(V).
$$
\end{proposition}

\begin{remark}
In our setting, it can also be shown that $v\in c^{1,\alpha}(\overline V)$.  In view of Remark \ref{whldmn},  one can find such $v\in c^{1, \alpha}(\Omega)$ provided $\Omega$ is  simply connected. We will not need these facts.
\end{remark}

\begin{proof}
 Let $A^{\e}$ be as defined in \eqref{Ae} and let $F^\e$ be the solution to the Neumann problem 
$$
\left \{ \begin{array}{ll} \Delta F^\e= {\rm div}\, A^{\e} & \mbox{in}\quad V \\ \partial_\nu F^\e =A^\e \cdot \nu  & \mbox{on} \quad \partial V
\end{array} \right.
$$ Since ${\rm div} (\nabla F^\e - A^{\e}) =0$, there exists a vector field $E^\e$ such that $A^\e = \nabla F^\e + \nabla ^\perp E^\e$.  $E^\e$ solves 
$$
\left \{ \begin{array}{ll} \Delta E^\e= {\rm curl} (A^\e - \nabla F^\e) = {\rm curl}\, A^\e 
& \mbox{in}\quad V \\ E^\e = const.  & \mbox{on} \quad \partial V
\end{array} \right.
$$ Since $A^\e$ is symmetric, we have $\partial_2 (F_1^\e+ E_2^\e) = \partial_1 (F_2^\e - E_1^\e)$, and 
thus we derive the existence of a scalar function $v^{(\e)} \in C^\infty(V)$ satisfying
$$
\partial_1 v^{(\e)} = F_1^\e+ E_2^\e,   \quad \partial_2 v^{(\e)} =  F_2^\e - E_1^\e, 
$$ i.e.\@ $\nabla v^{(\e)} = F^\e + (E_1^\e,  -E_2^\e)$. As a consequence
$$
A^\e = \nabla ^2 v^{(\e)} - \nabla  (E^\e)^\perp + \nabla^\perp E^\e. 
$$

Standard elliptic estimates  imply that for any $p<\infty$
\bee\label{errest}
\|\nabla E^\e\|_{W^{1,p}(V)} \le c(p,V) \|{\rm curl}\, A^\e\|_{0;V}. 
\eee 
 As a consequence, fixing $p>2$, and in view of \eqref{curle0}, $\|A^\e - \nabla^2 v^{(\e)}\|_{0;V} \le C o(\e^{3\alpha-2}) \to 0$ and therefore  $\nabla^2 v^{(\e)}$  converges to $A$ in the sense of distributions.
 
 Now, for all $\varphi \in C^\infty _c (V)$, there  exists a vector field   $\Psi $ in $V$, vanishing on $\partial V$, such that  ${\rm div} \, \Psi = \varphi - \fint_V \varphi$, for which 
$\|\Psi\|_{W^{1,2}(V)} \le c(V) \|\varphi\|_{L^2(V)}$ (see e.g.\@ \cite{ASV}).  By adjusting the $\nabla v^{(\e)}$ so that is it is of average 0 over $V$,  and in view of  Proposition \ref{2nd}, we obtain that for all $\e, \e' \le \delta$:
 
 $$
 \begin{aligned}
 \displaystyle \Big | \int_V (\nabla v^{(\e)} - \nabla v^{(\e')})  \varphi  \Big|  & =  \displaystyle \Big | \int_V (\nabla v^{(\e)} - \nabla v^{(\e')})  (\varphi  - \fint_V \varphi)  \Big|  
 = \Big | \int_V (\nabla v^{(\e)} - \nabla v^{(\e')})   {\rm div} \, \Psi  \Big | \\  & = \Big | \int_V (\nabla^2 v^{(\e)} - \nabla^2 v^{(\e')})   \Psi  \Big |  \\ & \le 
 \Big | \int_V (A^\e -A^{\e'})   \Psi  \Big |  + \Big | \int_V (\nabla^2 v^{(\e)} - A^{\e})   \Psi  \Big |    + \Big | \int_V (\nabla^2 v^{(\e')} -A^{\e'})   \Psi  \Big |   
 \\ & \le C \|\Psi\|_{W^{1,1}} o(\delta^{3\alpha-2})   \le C  \|\varphi\|_{L^2(V)}  o(\delta^{3\alpha-2}) . 
 \end{aligned}
 $$ By duality, we conclude that $\nabla v^{(\e)}$ converges strongly in $L^2(V)$ to a vector field $F$ and that $\nabla F = A$. It is now straightforward to see that, 
 if necessary by adjusting the $v^{(\e)}$ by constants, $v^{(\e)}$ converges strongly in 
 $W^{1,2}$ to  some $v\in W^{1,2}(V)$ and that 
 $A= \nabla^2 v$.

To complete the proof, it is only necessary to show that (a) $\Det D^2 v =0$ and (b) $v\in c^{1,\alpha}(V)$.  

\noindent In order to show (a), we first observe that since $\nabla v^{(\e)} \to \nabla v$ in  $L^2(V)$, for all $\psi \in C^\infty_c(V)$ we obtain:
$$
\begin{aligned}
\displaystyle -\frac 12{\rm curl}^T {\rm curl} (\nabla v \otimes \nabla v) [\psi] & = \lim_{\e\to 0} -\frac 12 {\rm curl}^T {\rm curl} (\nabla v^{(\e)} \otimes \nabla v^{(\e)}) [\psi]  = \lim_{\e \to 0}  {\rm det}\nabla^2 v^{(\e)} [\psi] \\ & 
=  \lim_{\e \to 0}   \int_V  {\rm det} (A^\e + \nabla (E^\e)^\perp - \nabla^\perp E^\e ) \psi \\ & = 
\lim_{\e \to 0}    \int _V (\det A^\e) \psi   +  \lim_{\e \to 0}    \sum_{\mbox{some indices } i,j,k,l}  \int_{V} A_{ij}^\e (\partial_k E_l^\e)\psi \\ & +  \lim_{\e \to 0}    \int_V {\rm det} 
(\nabla (E^\e)^\perp - \nabla^\perp E^\e) \psi .
\end{aligned}
$$  The third term obviously converges to $0$, and so does the  first term by Proposition \ref{Adet0}. For the second term, we have, using Proposition \ref{2nd}, \eqref{curle0} and \eqref{errest} 
 \begin{equation*}
 \begin{aligned}
\Big |  \int_V  A_{ij}^\e (\partial_k E_l^\e) \psi \Big | & \le \Big | (A_{ij}^\e - A_{ij}) [(\partial_k E_l^\e) \psi] \Big | + \Big | A [(\partial_k E_l^\e) \psi]  \Big |  \\ & 
\le C   \|(\partial_k E_l^\e) \psi\|_{W^{1,1}}o(\e ^{2\alpha-1}) + C  \|(\partial_k E_l^\e) \psi\|_{W^{1,1}}   \\ &  \le C \|\psi\|_1 o(\e^{2\alpha-1}) o(\e^{3\alpha-2}) + C \|\psi \|_1 o(\e^{3\alpha-2}) \xrightarrow{\e\to 0}0.
\end{aligned}
\end{equation*} This completes the proof of (a).

\medskip 

It remains to prove (b) $v\in c^{1,\alpha}(V)$.  Since $\nabla^2 v =A$ is symmetric in $i,j$, we have $\partial_{ji} v  =  \partial_{ij}v =  A_{ij}$, as distributions.  Proposition \ref{2nd} implies that, for each $i$, $\nabla (\partial_i v)$ satisfies the fine $\alpha$-interpolation inequality on $W^{1,1}_0(V)$. For any $x\in V$, we apply Proposition \ref{C-alpha-control} to a coordinate rectangular box containing $x$ and compactly included in $V$ to 
conclude that $\partial_i v \in c^{0,\alpha} (V)$ for $i=1,2$, which yields   $v\in c^{1,\alpha}(V)$.
\end{proof}

 \section{Developability: A proof  of Theorem \ref{iso23}}\label{thmproof}

Let $u\in c^{1,\alpha}(\Omega, \R^3)$ be an isometric immersion for $2/3 \le \alpha<1$. In order to prove our main Theorem, and in view of Proposition \ref{loc=glo}, it is sufficient to show that  $\nabla u$ satisfies condition (c).  We first fix an open disk  $V$ containing $x$ and compactly contained in  $\Omega$ and note the existence of the function $v\in  c^{1,\alpha}(V) \subset c^{1,2/3}(V) $ as defined in Proposition \ref{AHess}. 
We apply the  key developability result Theorem \ref{mp23dvp} to $v$ to obtain that $\nabla v$ satisfies any of the  equivalent conditions of Proposition \ref{loc=glo} in $V$.  The developability of $u$ is a consequence of Corollory \ref{loc=glo-really} and Proposition \ref{u-dev-v} below. The second conclusion follows from  Corollary \ref {be>al}.

\begin{proposition}\label{u-dev-v}
Let $1/2 < \alpha <1$  and $u\in c^{1,\alpha}(\Omega, \R^3)$ be an   immersion, with the second fundamental form $A$ defined as in  
Proposition \ref{2nd} on $V \Subset \Omega$. Assume moreover that $A= \nabla^2 v$, where $v\in C^1(V)$.   Then $u|_V$ is developable if and only if $v$ is developable.
\end{proposition} 
\begin{proof}
We consider the following  identity \cite[Equation (2.1.3)]{hanhong} (also known as the Gauss equation in the literature), which is valid for the mollified sequence of smooth immersions $u_\e$:
\begin{equation}\label{uij}
\partial_{ij} u_\e= \Gamma^{k,\e}_{ij} \partial_k u_\e + A^\e_{ij} N^{\e}, 
 \end{equation}  where the standard Einstein summation convention is used.  
 
Let $B\subset V$ be a given disk,   $\psi\in C^\infty_c(B)$ and Lipschitz unit vector field $\vec \eta :  B \to \R^2$. 
We denote the coefficients of $u$ and $N^\e$ respectively by $u^m$ and $N^{\e, m}$ and  calculate for $m=1,2,3$,   \begin{equation}\label{sim}
 \begin{aligned}
\int_B {\rm div} (\psi \vec \eta) (\nabla u^m)  & =   \lim_{\e \to 0} \int_B {\rm div} (\psi \vec \eta) (\nabla u^m_\e)    
=  \lim_{\e \to 0}  -\int_B  \psi \vec \eta  \cdot  \nabla^2 u^m_\e  
  \\ &  = \lim_{\e \to 0}    -\int_B  \psi \vec \eta \cdot  \Big (\Gamma^{k,\e} \partial_{k} u^m_\e + A^\e N^{\e,m}  \Big) 
  = \lim_{\e \to 0} - \int_B  \psi \vec \eta \cdot ( A^\e N^{\e,m} ),
 \end{aligned}
 \end{equation}  where we used  \eqref{Cffl}.   We have
  $$
 \begin{aligned}
 \displaystyle \int_B  \psi \vec \eta \cdot \Big (A^\e N^{\e,m} \Big )    & =  A^\e [N^{\e,m} \psi \vec \eta]   =   (A^\e - A)[N^{\e,m} \psi \vec \eta]  + A[N^{\e,m} \psi \vec \eta].
\end{aligned}
 $$  By Proposition \ref{2nd} we have: 
 $$
 \begin{aligned}
 \displaystyle \Big | (A^\e - A)[N^{\e,m} \psi \vec \eta] \Big |   &     \le  C \Big  (    \|N^{\e, m} \psi \vec \eta\|_{L^1}o(\e^{2\alpha-1}) +  
  \|\nabla (N^{\e, m} \psi \vec \eta)\|_{L^1} o(\e^{\alpha})   \Big ) \\ & 
 \le C   \|\psi\vec \eta\|_{L^1}o(\e^{2\alpha-1}) +  \|\nabla (\psi \vec \eta)\|_{L^1} o(\e^{\alpha})   \to 0 \quad \mbox{as} \quad \e \to 0,
 \end{aligned}  $$ where we used \eqref{estimN}. Note also that in view of $A= \nabla^2 v$,  we have for all $\e$, in view of the fact that $N^{\e,m}  \psi  \in C^\infty_c(B)$: 
 $$
 A[N^{\e,m} \psi \vec \eta] = -\int_B {\rm div} (N^{\e,m}  \psi \vec \eta) \nabla v.
 $$  Therefore, \eqref{sim} implies
\begin{equation}\label{u-rel-v}
 \int_B {\rm div} (\psi \vec \eta) \nabla u^m= \lim_{\e\to 0} \int_B {\rm div} (N^{\e,m}  \psi \vec \eta) \nabla v.
 \end{equation}

If $v$ is developable, we apply  Corollary \ref{lipvec}  to $\nabla v$,  and choose accordingly  the open disk $B$ around  $x\in V$  and the Lipschitz unit vector field $\vec \eta : B \to \R^2$. Therefore   
\begin{equation}\label{v-eta}
\displaystyle \forall y\in B\quad   \vec \eta (y) = \vec \eta (y+ s \vec \eta (y))  \,\,\, \mbox{{\rm and}}  \,\,\, 
    \nabla v  (y) = \nabla v (y+ s \vec \eta (y)) \,\,\, \mbox{{\rm for all}} \,\, s \,\, \mbox{{\rm for which}} 
     \,\,\, y+ s  \vec \eta (y) \in B. 
\end{equation}  We apply Lemma \ref{weakconst} to $f=\nabla v$ and $B= B_x$ to obtain for any $\psi\in C^\infty_c(B)$:
 $$
\int_B   \, {\rm div} (N^{\e,m} \psi \vec \eta) \nabla v =0,\quad  m=1,2,3,
$$
where we used  the fact that $ N^{\e,m} \psi \in C^\infty_c(B)$.  By \eqref{u-rel-v} we conclude that   
$$
\int_B {\rm div} (\psi \vec \eta) \nabla u = 0.
$$ Once again applying Lemma \ref{weakconst} and \eqref{v-eta} implies that the Jacobian derivative  $\nabla u$ of the isometric immersion $u$ is constant along the segments generated by the vector field $\vec \eta$ in $B$. Therefore  $\nabla u$    satisfies condition (c) of Proposition \ref{loc=glo}, which implies the developability of $u|_V$.  As already mentioned, this part of the proof concludes the proof of Theorem \ref{iso23} for $u\in c^{1,\alpha}$, $2/3 \le \alpha<1$.

To finish the proof of Proposition \ref{u-dev-v} we need to prove the converse statement. If, on the other hand, $u|_V$ is assumed to be developable, we proceed in a similar manner. According to Corollary \ref{lipvec} and Lemma \ref{weakconst}, for all $x\in V$ there exist a disk $B\subset V$     centered at $x$ and a unit Lipschitz vector field $\vec \eta : B \to \R^2$ such that 
for all $ \psi \in C^\infty_c(B)$:  
\begin{equation*}\label{u-eta}
\displaystyle \forall y\in B\quad  \vec \eta (y) = \vec \eta (y+ s \vec \eta (y))  \,\,\, \mbox{{\rm and}}  \,\,\, 
  \int_B    {\rm div} (N^{\e,m} \psi  \vec \eta)  \nabla  u^m   =0, m=1,2,3.    
\end{equation*} Hence, using the fact that $A_{ij}^\e = \partial_{ij} u_\e \cdot N^\e$ we  deduce for all $\psi \in C^\infty_c(B)$ : 
$$
\begin{aligned}
   \int_B {\rm div} (\psi \vec \eta) (\nabla v) & =  -  A [\psi \vec \eta]  =  \lim_{\e \to 0}  A^\e [ \psi \vec \eta]  = \lim_{\e \to 0}   \int_B    \sum_{m=1}^3 (N^{\e,m} \nabla^2  u^m_\e)  \psi  \vec \eta \\ & = \lim_{\e \to 0}  - \int_B    \sum_{m=1}^3  {\rm div} (N^{\e,m} \psi  \vec \eta)   \nabla  u^m_\e   \\ & =   \lim_{\e \to 0}   - \int_B    \sum_{m=1}^3   {\rm div} (N^{\e,m} \psi  \vec \eta) (\nabla  u^m_\e - \nabla u^m) = 0,
\end{aligned}
$$  where we used \eqref{estims} and \eqref{estimN} to obtain a convergence of order $o(\e^{2\alpha-1})$ in the last line. Following the same line of argument as before $v$ is hence developable in view of Lemma \ref{weakconst}.  \end{proof} 
\appendix
\renewcommand{\thesection}{\Roman{section}} 

\section{Developabiltiy of each component can be shown independently}\label{component}
  
   \begin{proposition}\label{compdev}
   Assume $\Omega\subset \R^2$ is an arbitrary domain, $ 2/3 \le \alpha <1$  and let $u\in c^{1, \alpha} (\Omega, \R^3)$ be an isometric immersion.    Then for each $m=1,2,3$, the component $u^m$ is developable in $\Omega$.    
   \end{proposition}

   \begin{remark}
Obviously this statement does not guarantee that the constancy lines or regions of $\nabla u^m$ are the same for the three components. If this fact is independently shown (e.g.\@ through geometric arguments)  Theorem \ref{iso23} will also be 
proved through this approach. 
 \end{remark} 
 \begin{proof}
 In view of Thereom \ref{mp23dvp}, it suffices to show that $\Det D^2 u^m=0$  for each $m$. Once again, we consider the mollifications $u_\e = u\ast \phi_\e$ defined on a suitably chosen domain $V\Subset \Omega$ and use \eqref{uij} to calculate for  m=1,2,3 and $\psi \in C^\infty_c(V)$:
  
 $$
\begin{aligned}
\displaystyle -\frac 12{\rm curl}^T {\rm curl} (\nabla u^m \otimes \nabla u^m) [\psi] & = 
\lim_{\e\to 0} -\frac 12 {\rm curl}^T {\rm curl} (\nabla u^m_\e \otimes \nabla u^m_\e) [\psi]  = \lim_{\e \to 0}  {\rm det}\nabla^2 u^m_\e [\psi] \\ & 
=  \lim_{\e \to 0}   \int_\Omega  {\rm det} (\Gamma^{k,\e} \partial_k u^m_\e+  A^\e N^{\e,m}  ) \psi \\ & = 
\lim_{\e \to 0}  \int _\Omega \det (A^\e) (N^{\e,m})^2 \psi  +  \lim_{\e \to 0}    \int_\Omega {\rm det} (\Gamma^{k,\e} \partial_k u^m_\e) \psi \\ & 
+ \lim_{\e \to 0}    \sum_{\mbox{some indices}\,\, i,j,r,s}  \int_{\Omega} A_{ij}^\e N^{\e,m} (\Gamma^{k,\e}_{rs} \partial_k u_\e^m )\psi.
\end{aligned} 
$$  The second term obviously converges to $0$ in view of \eqref{Cffl}. 
For the third  term we use \eqref{estims}, \eqref{estimA} and  \eqref{Cffl} to prove:
$$
\Big |\int_{\Omega} A_{ij}^\e N^{\e,m} (\Gamma^{k,\e}_{rs} \partial_k u_\e^l )\psi  \Big | \le C \|A^\e\|_{0;V} \|\Gamma^{k,\e}\|_{0;V}  \|\psi\|_{L^1}  \le C   \|\psi\|_{L^1} o(\e^{3\alpha-2}) \to 0. 
$$ Finally, the vanishing of the first term follows from   Propostion \ref{Adet0} and \eqref{estimN} since
$$
\begin{aligned}
\displaystyle \Big |\int _\Omega \det (A^\e) (N^{\e,m})^2 \psi  \Big | & \le C \|(N^{\e,m})^2 \psi \|_{W^{1,1}} o(\e^{2\alpha-1}) \\ &
\le C   \Big ( \|\psi\|_{W^{1,1}} o( \e^{2\alpha-1}) + \|\nabla N^\e\|_{0;V} \|\psi\|_{L^1} o(\e^{2\alpha-1}) \Big )  \\ &  \le C   \|\psi\|_{W^{1,1}} o(\e^{3\alpha-2}) \to 0.   
 \end{aligned}   
$$ \end{proof}

  \section{Little H\"older spaces and modulus of continuity}\label{litmod}
   
Here we will prove  two auxiliary statements regarding little H\"older spaces.
  
\begin{lemma}\label{litmolapp}
Let $V\subset U\subset \R^n$ be open sets such that ${\rm dist}{(\overline V, \partial U)} >0$ and $f: U\to \R$. If for $0<\alpha<1$, $[f]_{0,\alpha;U|r}$ is $o(1)$ as a function of $r$, 
then $\|f_\e - f\|_{0,\alpha;V}\to 0$  as $\e \to 0$.  
\end{lemma}

\begin{proof}

Let $\e < {\rm dist}{(\overline V, \partial U)}$. We first note that by Lemma \ref{estimcvls}-(i) we have
\begin{equation}  \label{semiest} 
 \|f_\e - f\|_{0;V}  \le  c[f]_{0,\alpha;U|\e} \e^\alpha.  
\end{equation}  We therefore estimate 
\begin{equation*} 
\begin{array}{ll}
[f_\e -f]_{0,\alpha; V} &  \le [f_\e - f]_{0, \alpha; V |\e} +  
\displaystyle \sup_{\begin{array}{c} x,y \in V \\   \e < |x-y| \end{array}} 
\frac{|f_\e(x) - f(x)  + f_\e (y)- f(y)|} {|x-y|^\alpha} 
\\ & \\ & \le  [f_\e - f]_{0, \alpha;V|\e} +  c\|f_\e - f\|_{0;V} \e^{-\alpha} \\ & 
\le  [f_\e]_{0, \alpha; V |\e} + [f] _{0, \alpha;V|\e}  + c [f]_{0,\alpha;U|\e} \\ &  \le [f_\e]_{0, \alpha; V |\e} + c  [f]_{0,\alpha;U|\e}. 
\end{array}
\end{equation*}   We now observe that 
$$
\begin{array}{ll}
 [f_\e] _{0, \alpha;V|\e} & =  \displaystyle  \sup_{\begin{array}{c} x,y \in  V \\   0 < |x-y| \le \e \end{array}} 
\frac{\Big |\displaystyle \int_{\R^n} (f (x-z)-  f (y-z)) \phi_\e (z) \, dz\Big |}{\displaystyle |x-y|^\alpha}  
\\ &  \le   \displaystyle  \sup_{\begin{array}{c} x,y \in  V \\   0 < |x-y| \le \e \end{array}} 
\displaystyle \int_{\R^n}  \frac{|f (x-z)-  f (y-z)|}{\displaystyle |x-y|^\alpha}| \phi_\e (z)|  \, dz 
\\ & \le c [f]_{0, \alpha;U|\e}.    
\end{array}
$$
Combining the latter two estimates  we hence obtain
$$
[f_\e -f]_{0,\alpha; V} \le  c [f]_{0,\alpha;U|\e}. 
$$ In view of \eqref{semiest} and the main assumption, we conclude with the  desired approximation of $f$ by $f_\e \in C^\infty(\overline V)$ in $C^{0,\alpha}$.
\end{proof}

\begin{proposition}\label{extension} 
Let $U\subset\R^n$ be a bounded weakly Lipschitz domain and let $f: U \to \R$ be uniformly continuous. Then $f$ admits an extension $\tilde f: \R^n \to \R$ with  $\omega_{\tilde f;\R^n}\le c(U) \omega_{f;U}$.   
\end{proposition}

\begin{proof}   $U$ is a bounded weakly Lipschitz domain, i.e.\@ by the standard definition,  there is a covering of 
$\overline U$ of open sets $U_i$ and charts $\psi_i: U_i \to \R^n$, $0\le i\le N$ such that $U_0 \subset U$  and for $0 \le i\le N$,
$\psi_i $ is a bilipschitz diffeomorphism between $U_i$ and an open cylinder $V_i:= B_i \times (-1,1) \subset \R^n$ such that $\psi_i(U\cap U_i)= V_i^+ := B_i \times  (0,1)$ and $\psi_i (\partial U \cap U_i) = B_i \times \{0\}$, where $B_i \subset\R^{n-1}$ is the unit ball.  

For each $1 \le i\le N$, consider the function $h_i:= f \circ \psi^{-1} : V_i^+ \to \R$. We can estimate the modulus of continuity of each $h_i$ from above by the modulus of continuity of $f$ using the maximum bilipschitz constant of the $\psi_i$, which depends on $U$:   
 $$
\omega_i:= \omega_{h_i; V_i^+} \le c(U) \omega_{f; U}.  
$$
Note that $\omega_i$ is necessarily nondecreasing by definition. Since each $V_i^+$ is convex,   $\omega_i$ is also subadditive, i.e.\@ 
for all $r_1, r_2\ge 0$, 
$$
\omega_i(r_1+ r_2) \le \omega_i (r_1) + \omega_i(r_2).
$$ Moreover the uniform continuity of $f$  implies that $$\lim_{r\to 0} \omega_i (r)=0.$$ A direct application of \cite[Theorem 13.16]{WW} implies  that  $h_i$ admits an extension $\tilde h_i$, given by
$$
\tilde h_i(x) := \sup \{h_i(y) - \omega_i (|x-y|); \,\, y\in V_i^+\},
$$ to the whole $\R^n$ also satisfying  
$$
\omega_{\tilde h_i; \R^n} \le \omega_i \le c(U) \omega_{f; U}.  
$$  Let 
$\{\theta_i\}_{0 \le i\le N}$ be a partition of unity corresponding to the covering $U_i$, i.e.\@ 
$$
\forall \, 0 \le  i \le N \quad \theta_i \in C^\infty_c (U_i, [0,1]) \quad \mbox{and}\quad  \sum_{i=0}^N \theta_i =1.  
$$   We extend $\theta_0 f :U_0 \to \R$ to $\tilde  f_0:\R^n \to \R$, and $\theta_i (\tilde h_i \circ \psi_i) : U_i \to \R$  to $\tilde f_i: \R^n \to \R$  by the trivial zero extension and  write 
$$
\tilde f:= \sum_{i=0}^N \tilde f_i. 
$$ The conclusion follows.   
\end{proof}

\end{document}